\documentclass[11pt,a4paper]{amsart}

\usepackage{amssymb}
\usepackage{amsmath,amsthm,geometry,enumerate}
\usepackage{graphicx}
\usepackage{psfrag}
\usepackage{amscd}
\usepackage[all]{xy}
\usepackage{MnSymbol}
\usepackage{booktabs}

\DeclareMathOperator{\pic}{Pic}
\DeclareMathOperator{\picrel}{PicRel}

\DeclareMathOperator{\pr}{PR}

\DeclareMathOperator{\Hom}{Hom}
\DeclareMathOperator{\ShHom}{\mathscr{H}\text{\kern -3pt {\calligra\large om}}\,}
\DeclareMathOperator{\Aut}{Aut}

\DeclareMathOperator{\spec}{Spec}

\newcommand{\Mm}{\mathcal{M}}
\newcommand{\Mb}{\overline{\mathcal{M}}}
\newcommand{\Cb}{\overline{\mathcal{C}}}
\newcommand{\C}{\overline{\mathcal{C}}}

\newcommand{\Jb}{\overline{\mathcal{J}}}
\newcommand{\J}{{\mathcal{J}}}

\newcommand{\Pic}{\operatorname{Pic}}
\newcommand{\Ker}{\operatorname{Ker}}

\newcommand{\GSym}{\operatorname{GSym}}
\newcommand{\OD}{\mathcal{O}(\mathcal{D})}
\newcommand{\ODF}{\mathcal{O}(\mathcal{D}(\phi))}

\theoremstyle{plain}

\newtheorem{theorem}{Theorem}
\newtheorem*{theorem*}{Theorem}
\newtheorem{proposition}{Proposition}[section]
\newtheorem{fact}{Fact}
\newtheorem{corollary}[proposition]{Corollary}
\newtheorem{lemma}[proposition]{Lemma}
\theoremstyle{definition}
\newtheorem{definition}[proposition]{Definition}
\newtheorem{remark}[proposition]{\textbf{Remark}}
\newtheorem{example}[proposition]{\textbf{Example}}

\begin{document}

\title{The stability space of compactified universal Jacobians}

\author{Jesse Leo Kass}

\address{J.~L.~Kass, Dept.~of Mathematics, University of South Carolina, 1523 Greene Street, Columbia, SC 29208, United States of America}
\email{kassj@math.sc.edu}
\urladdr{http://people.math.sc.edu/kassj/}

\author{Nicola Pagani}

\address{N.~Pagani, Department of Mathematical Sciences, University of Liverpool, Liverpool, L69 7ZL, United Kingdom}
\email{pagani@liv.ac.uk}
\urladdr{http://pcwww.liv.ac.uk/~pagani/}
\date{\today}

\begin{abstract}
	In this paper we describe compactified universal Jacobians, i.e.~compactifications of the moduli space of line bundles on smooth curves obtained as moduli spaces of rank~$1$ torsion-free sheaves on stable curves, using an approach due to Oda--Seshadri.  We focus on the combinatorics of the stability conditions used to define compactified universal Jacobians.  We explicitly describe an affine space, the stability space, with a decomposition into polytopes such that each  polytope corresponds to a  proper Deligne--Mumford stack that compactifies the moduli space of line bundles.  We apply this description to describe the set of isomorphism classes of compactified universal Jacobians (answering a question of Melo), and to resolve the indeterminacy of the Abel--Jacobi sections (addressing a problem raised by Grushevsky--Zakharov).
\end{abstract}

\maketitle

\section{Introduction}
In this paper we study the problem of extending the universal Jacobian $\J^{d}_{g, n}$ over the moduli space of smooth $n$-pointed curves of genus $g$ to a proper family over the moduli space $\Mb_{g, n}$ of stable pointed curves.  Recall that $\J^{d}_{g, n}$ is the moduli space of degree $d$ line bundles on smooth curves. We extend it as a moduli space of sheaves.  One extension of $\J_{g, n}^d$ is the moduli space $\operatorname{Simp}_{g, n}^d$ of all simple rank~$1$ torsion-free sheaves of degree $d$, but this extension fails to be proper.  Indeed, while it satisfies the existence part of the valuative criterion of properness \cite[Theorem~32]{esteves}, it is not proper because it fails to be separated and of finite type.

Rather than working directly with $\operatorname{Simp}_{g, n}^{d}$, we analyze extensions of $\J^d_{g, n}$ that are  suitable  proper subspaces (or substacks) of $\operatorname{Simp}_{g, n}^d$.  The proper subspaces  of $\operatorname{Simp}_{g, n}^{d}$ we describe are the subspaces defined by choosing a set of multidegrees for each curve and taking the subspace of $\operatorname{Simp}_{g, n}^{d}$ parameterizing the sheaves with multidegree equal to one of the chosen multidegrees.    Here the multidegree of a line bundle $L$ on a reducible curve is the vector whose components are the degrees of the restrictions of $L$ to the irreducible components of the curve.  The problem of prescribing a collection of multidegrees with the property that the resulting subspace of $\operatorname{Simp}_{g, n}^{d}$ is a proper extension of $\J_{g, n}^{d}$ has been studied by a large number of authors; see e.g.~\cite{oda79, altman80, caporaso, simpson, panda, esteves, caporaso08a, melo09, melo11, melo}.  In this paper, we introduce and study subspaces of $\operatorname{Simp}_{g, n}^{d}$ produced by generalizing to the universal family of curves an approach developed by Oda--Seshadri.  With our construction, we produce the commonly studied spaces that extend $\J^{d}_{g, n}$ to a proper space.  In particular, our construction recovers the moduli spaces constructed by Melo in \cite{melo}.  We explain the relation with Melo's work in  Remark~\ref{esteves} and with other work in Remarks~\ref{Remark: RelationWithCaporasoPand} and \ref{Remark: RelationWIthGrushevsky}.

In  \cite{oda79} Oda--Seshadri introduced, for a nodal curve $C$ (and $d=0$), the  \emph{stability space} $V^d(C)$ as the affine space of functions $\phi \colon \{ C_i \subseteq C \text{ an irreducible component}\}~\to~\mathbb{R}$ such that $\sum \phi(C_i)=d$.  For a \emph{nondegenerate} $\phi \in V^d(C)$ (i.e.~a $\phi$ not lying in a certain collection of affine hyperplanes) they proved that the moduli space of $\phi$-semistable sheaves, i.e.~sheaves whose multidegree is sufficiently close to $\phi$ (in a sense that we make precise in Definition~\ref{phistab}), is a proper subspace $\overline{J}_C(\phi)$ of the space of simple sheaves on $C$, which  we call a (fine) $\phi$-compactified Jacobian.

We extend Oda--Seshadri's approach to describe moduli spaces over $\Mb_{g, n}$.  In Section~\ref{Section: StabilitySpace} we construct a space $V^{d}_{g, n}$ of stability conditions for the universal stable pointed curve.  The affine space $V^d(C)$ associated to a nodal curve $C$ depends only on the dual graph $\Gamma_C$ of $C$.  Denoting by ${G}_{g,n}$ the set of isomorphism classes  of stable $n$-marked graphs of genus $g$, we define \[V_{g, n}^d \subseteq \prod_{\Gamma \in {G}_{g,n}} V^d(\Gamma)\] as the subspace consisting of those vectors $\phi =  ( \phi(\Gamma) \in V^d(\Gamma))_{\Gamma \in {G}_{g,n}}$ that satisfy a compatibility condition with respect to automorphisms and contractions of the dual graphs (see Definition~\ref{stabilityspacedef} for details).  For a \emph{nondegenerate} $\phi \in V_{g, n}^d$, we show that \[\{ \text{$\phi$-stable sheaves on stable curves}  \} \subseteq \operatorname{Simp}_{g, n}^{d}\] is a proper moduli space that we call a (fine) $\phi$-compactified universal  Jacobian $\Jb_{g,n}(\phi)$. We give the precise definition of $\Jb_{g,n}(\phi)$  in Section~\ref{Section: construction}.  When $\phi$ is nondegenerate, we show in Corollary~\ref{Cor: JbExists} that $\Jb_{g, n}(\phi)$ is a proper Deligne--Mumford stack, a result we deduce from Simpson's representability result \cite[Theorem~1.21]{simpson}.

The main result about the stability space is Theorem~\ref{pic=stab}, where we describe $V_{g, n}^d$ as the degree $d$ subspace of the real relative Picard group of the universal curve on $\Mb_{g,n}$.  As a by-product, in Corollary~\ref{isoalphan} we prove that an element $\phi \in V_{g,n}^d$ is uniquely determined by its components $\phi(\Gamma)$ for $\Gamma$ the dual graph of certain stable pointed curves with $2$ smooth irreducible components and at most $2$ nodes, so that in particular, two extensions $\Jb_{g,n}(\phi_1)$  and $\Jb_{g,n}(\phi_2)$ that coincide in codimension $2$ must be equal.

The main contribution of this paper is the description in Section~\ref{Section: Walls} of how the moduli spaces $\Jb_{g,n}(\phi)$ depend on $\phi$.  There we define $\phi_1$ to be equivalent to $\phi_2$ when $\phi_1$-stability coincides with $\phi_2$-stability.  The equivalence classes are the interiors of  rational bounded convex polytopes in $V_{g,n}^d$ that we call \emph{stability polytopes}.  We then exhibit in Theorem~\ref{stabilityspace} an explicit set of equations for the defining hyperplanes.
\begin{theorem*}
	For $g \ge 2, n \geq 1$ and $N = N(g, n)$ the number of boundary divisors in $\Mb_{g, n}$, there is an explicit isomorphism (Corollary~\ref{isoalphan}) of affine spaces $V^{g-1}_{g, n}\cong \mathbb{R}^{N-1} \times \mathbb{R}^n$. The decomposition of $V_{g,n}^{g-1}$ into stability polytopes is the product of the decomposition of $\mathbb{R}^{N-1}$ by integer translates of coordinate hyperplanes and the decomposition of $\mathbb{R}^{n}$ by integer translates of the following hyperplanes
	\begin{equation} \label{Eqn: SimplifiedHyperplanes}
		\left\{ \vec{x} \in \mathbb{R}^{n} \colon \ \sum_{i \in S} x_{i} - \frac{\ell}{2g-2}\sum_{i=1}^{n} x_{i} =0 \right\} \text{ for $\ell = 0, \dots, 2g-3, \ \emptyset \subsetneq S \subseteq \{ 1, \dots, n\}$.}
	\end{equation}
\end{theorem*}
This is  Theorem~\ref{stabilityspace} in the special case when $d=g-1$ (and $g \geq 2, n \geq 1$).    When $d \ne g-1$, the decomposition of $V^{d}_{g,n}$ is similar but  the hyperplanes are translated.  The factors in the product decomposition correspond to two vector spaces $C_{g,n}\cong \mathbb{R}^{N-1}$ and $D_{g,n} \cong \mathbb{R}^n$ that we introduce in  Definition~\ref{CDT}.

The hyperplanes in \eqref{Eqn: SimplifiedHyperplanes} with $\ell=0$ are known as the resonance hyperplanes in the literature, so \eqref{Eqn: SimplifiedHyperplanes} defines a refinement of the resonance hyperplane arrangement.    The above description of the stability polytopes should be compared with a similar description in \cite{kasspa}.  There we carried out the analogous program for extensions of $\J^{g-1}_{g, n}$ over the moduli stack of treelike curves $\Mm_{g, n}^{\text{TL}} \subseteq \Mb_{g, n}$  using an affine space $V_{g, n}^{\text{TL}}$ analogous to $V^{d}_{g, n}$.  As we explain  in Remarks~\ref{compareoldpaper} and \ref{compareoldpaper2}, $V_{g, n}^{\text{TL}}$ is canonically isomorphic to  $C_{g, n}$. Theorem~\ref{stabilityspace} shows a considerable increase in the combinatorial complexity when passing from the problem of extending  $\J_{g, n}^{d}$ over $\Mm_{g,n}^{\text{TL}}$ to the problem of extending it over $\Mb_{g, n}$ as the resonance hyperplane arrangement is more complicated than the arrangement of coordinate hyperplanes.

The difference between $\Mm_{g,n}^{\text{TL}}$ and $\Mb_{g, n}$  is also demonstrated by the results in Section~\ref{Section: different}, where we describe how $\Jb_{g, n}(\phi)$ depends on $\phi \in V^{d}_{g,n}$.  Over treelike curves, we showed in \cite{kasspa} that, while changing $\phi$ changes  the set of $\phi$-stable sheaves, the corresponding Deligne--Mumford stacks $\Jb_{g, n}(\phi)|\Mm_{g,n}^{\text{TL}}$ are all isomorphic.  The situation over $\Mb_{g, n}$ is different.  We show
\begin{theorem*}
	When $\Mb_{g, n}$ is of general type, there exist nondegenerate $\phi_1, \phi_2 \in V_{g, n}$ such that $\Jb_{g, n}(\phi_1)$ and $\Jb_{g, n}(\phi_2)$ are not isomorphic as Deligne--Mumford stacks.
\end{theorem*}
This is Corollary~\ref{nonisomorphic}, and the result answers a question of Melo in \cite[Question~4.15]{melo} (see Remark~\ref{esteves} for a description of the relation of that work to this paper).

We deduce the result from Corollary~\ref{noisocommuting} which states that, for all $(g,n)$ with $g>0$ except for those in the finite list \eqref{list}, there exist nondegenerate $\phi_1$ and $\phi_2$ such that $\Jb_{g,n}(\phi_1)$ and $\Jb_{g,n}(\phi_2)$ are not isomorphic as Deligne--Mumford stacks over $\Mb_{g, n}$.  In fact, in Section~\ref{Section: different} we  show  that the isomorphism classes of  $\Jb_{g, n}(\phi)$'s, considered as stacks over $\Mb_{g, n}$, are in bijection with the quotient of the set $ \mathcal{P}_{g,n} $ of stability polytopes by the action of the generalized dihedral group $\widetilde{\pr}_{g,n}$ of the relative Picard group of the universal curve. From this analysis we also deduce that, for fixed $(g,n)$, there are finitely many non-isomorphic $\Jb_{g,n}(\phi)$ for all $d \in \mathbb{Z}$ and all nondegenerate $\phi \in V_{g,n}^d$. % $\pi \colon \Cb_{g,n} \to \Mb_{g,n}$.

In Section~\ref{Section: AbelJacobi} we give a second application of our description of $V^{d}_{g, n}$, namely a resolution of the indeterminacy of the Abel--Jacobi sections.  Recall that, given a vector  $(k; d_1, \ldots, d_n)$ of integers satisfying $k(2-2g) + d_1 + \ldots + d_n =d$, the rule
\begin{equation} \label{Eqn: AbelMap}
	(C, p_1, \ldots, p_n) \mapsto \omega^{\otimes -k}_C(d_1 p_1 + \ldots + d_n p_n)
\end{equation}
defines a morphism $\sigma_{k, \vec{d}} \colon \Mm_{g, n} \to \J^{d}_{g, n}$ and hence a rational map from $\Mb_{g, n}$ into any extension of $\J_{g, n}^{d}$.  Grushevsky--Zakharov raised the problem of resolving the indeterminacy of this map  in \cite[Remark~6.3]{grushevsky}.  In Corollary~\ref{Cor: AbelMapIndeterminacy},  we describe the locus of indeterminacy as
\begin{theorem*} %()
	For $\phi$ nondegenerate, the locus of indeterminacy of $\sigma_{k, \vec{d}} \colon \Mb_{g, n} \dashrightarrow \Jb_{g, n}(\phi)$ is the closure of the locus of pointed curves $(C, p_1, \ldots, p_n)$ that have $2$ smooth irreducible components meeting in at least $2$ nodes with the property that $\omega_C^{\otimes -k}(d_1 p_1 + \ldots + d_n p_n)$ fails to be $\phi$-stable.
\end{theorem*}
This result is consistent with earlier work of Dudin.  In  \cite[Section~3]{dudin}, Dudin proved that, for certain $\phi$,  the locus of indeterminacy of  $\sigma_{k,\vec{d}}$ is contained in the closure of the locus of pointed curves that consist of $2$ smooth curves meeting in at least $2$ nodes satisfying the above stability condition \cite[Proposition~3.3]{dudin}.  The main new content of the above theorem is that the containment of the indeterminacy locus is in fact an equality.  (For a detailed explanation of which $\Jb_{g,n}(\phi)$ Dudin studies, see the discussion immediately after \cite[Corollary~5.4]{kasspa} and Remark~\ref{esteves}.)
%Dudin proved in \cite{dudin} that, for certain nondegenerate $\phi \in V_{g,n}^d$ (see Remark~\ref{relationwithDudin} for more details), the locus of indeterminacy of $\sigma_{k, \vec{d}}$ is a subset of the closure of the locus that consists of two smooth curves meeting in $k \ge 2$ nodes with the property that $\mathcal{O}_C(d_1 p_1 + \ldots + d_n p_n)$ fails to be $\phi$-stable. Our Proposition~\ref{Prop: AbelMapIndeterminacy} says that this inclusion is in fact an equality.

The result can be described in terms of the \emph{degenerate} vector $\phi_{k,\vec{d}} \in V^{d}_{g, n}$ that is the multidegree of $\omega_C^{\otimes -k}(d_1 p_1 + \ldots + d_n p_n)$.  When $\phi$ is nondegenerate and sufficiently close to $\phi_{k,\vec{d}}$, our result states that the locus of indeterminacy is empty. For general $\phi$, the rational map $\sigma_{k,\vec{d}}\colon \Mb_{g,n}\dashrightarrow  \Jb_{g,n}(\phi)$ has indeterminacy that we can resolve as follows.  If $\phi_{0}$ is nondegenerate and sufficiently close to $\phi_{k,\vec{d}}$, then $\Jb_{g, n}(\phi)$ is related to $\Jb_{g, n}(\phi_{0})$ by a series of flips that correspond to the values of $t \in [0, 1]$ such that $t \phi_{0} + (1-t) \phi$ lies in the boundary of a stability polytope.  Indeed, the moduli spaces $\Jb_{g, n}(\phi)$ are  locally constructed   using GIT (through our use of \cite{simpson}), and the structure of these flips is described by Thaddeus in \cite{thaddeus}. The above theorem shows that the indeterminacy of $\sigma_{k,\vec{d}}$ is resolved by modifying $\Jb_{g, n}(\phi)$ by these flips.

The relation of this result to the work of  Grushevsky--Zakharov \cite{grushevsky} is complicated as they consider $\sigma_{k,\vec{d}}$ as a rational map into the extension $\mathcal{Y}'_{g,n}$ of $\J^{0}_{g, n}$ given by Mumford's rank~$1$ degenerations,  and this extension is different from, but related to, the $\Jb_{g, n}(\phi)$'s.  We discuss the relation with $\mathcal{Y}'_{g,n}$ in Remark~\ref{Remark: RelationWIthGrushevsky}.

%We discuss this relation in detail in  Remark~\ref{Remark: RelationWIthGrushevsky}. One important point is that we resolve indeterminacy by passing to a stack $\Jb_{g,n}(\phi_0)$ that often does not admit a morphism to $\mathcal{Y}'_{g,n}$. In Remark~\ref{Remark: RelationWIthGrushevsky}, we show that there are nondegenerate stability parameters $\phi_{\epsilon}$ with the property $\Jb_{g,n}(\phi_{\epsilon})$ admits a morphism to $\mathcal{Y}'_{g,n}$ (assuming  an expected comparison result that is not available in print;  see the 4th paragraph of the remark), but $\Jb_{g,n}(\phi_{\epsilon})$  is not equal to $\Jb_{g,n}(\phi_{0})$ in general.  Instead, these stacks are related by a sequence of flips,  and these flips may have indeterminacy.

Observe that  we resolve the indeterminacy by modifying the target (the compactified universal Jacobian) of the Abel--Jacobi section  rather than source (the moduli space of curves)  as is often done when resolving indeterminacy.  An approach to resolving indeterminacy by modifying the source is worked out by  David Holmes in  \cite{holmes}.  He analyzes $\sigma_{k,\vec{d}}$ when $k(2-2g) + d_1 + \ldots + d_n =0$ and produces a morphism from an open substack of an explicit toric blowup of $\overline{\mathcal{M}}_{g,n}$ into the separated stack parameterizing multidegree $0$ line bundles on stable curves.  Holmes uses this resolution to study the double ramification cycle, a topic we do not study here.  After this paper was submitted for publication, Marcus--Wise published \cite{marcus}  to the repository the arXiv.  In that paper, they analyze the Abel--Jacobi section in a manner similar to that of Holmes.

% Shortly before posting this work, we became aware of the preprint ``Extending the double ramification cycle by resolving the Abel--Jacobi map'' by David Holmes, which gives a very different approach to resolving the indeterminacy of the Abel-Jacobi section $\sigma_{\vec{d}}$ when $\sum d_i =0$, to the separated space $J$ that parametrizes line bundles on stable curves whose partial degree on every component is zero. The approach of Holmes, which is motivated by the quest of computing the Jacobian extension of the double ramification cycle, is by considering open sets of blow-ups of $\Mb_{g,n}$ so that the rational map $\sigma_{\vec{d}} \colon \Mb_{g,n} \to J$ lifts to a well-defined morphism from the open set to $J$.

\subsection{Organization of the paper}  In Section~\ref{background} we collect background material on moduli spaces of curves. In Section~\ref{graphs} we fix the notation for stable graphs, in Section~\ref{contractions} we define a notion of contraction, in Section~\ref{notationmoduli} we discuss the stratification of $\Mb_{g,n}$ by topological type and in Section~\ref{Picard} we describe the relative Picard group of the universal curve $\pi \colon \Cb_{g,n} \to \Mb_{g,n}$. In Section~\ref{Section: StabilitySpace} we introduce the universal stability space $V_{g,n}^d$ and prove two results that describe it explicitly: Theorem~\ref{pic=stab} and Corollary~\ref{isoalphan}. In Section~\ref{Section: construction} we define the stacks $\Jb_{g,n}(\phi)$ and prove that they are $k$-smooth Deligne--Mumford stacks when $\phi$ is nondegenerate. In Section~\ref{Section: Walls} we introduce the stability polytope decomposition $\mathcal{P}_{g,n}$ and prove Theorem~\ref{stabilityspace}, which gives an explicit description of the stability hyperplanes of the stability spaces $V_{g,n}^d$. In Section~\ref{final} we apply our results to resolve the indeterminacy of the Abel--Jacobi sections (Section~\ref{Section: AbelJacobi}) and to enumerate the different $\Jb_{g,n}(\phi)$ (Section~\ref{Section: different}). Section \ref{Section: Appendix} is the Appendix, where we collect some algebra lemmas needed in Section~\ref{final}.

\subsection{Conventions}
%Throughout the paper we work with natural numbers $g,n$ that satisfy $2g-3+n>0$.

We denote by $[n]$  the set $\{1, \ldots, n\}$. If $S \subseteq [n]$, we write $S^c$ for $[n] \setminus S$. For a given subset $S \subseteq [n]$ and $f \colon S \to \mathbb{Z}$, we denote by $f_S$ the sum $\sum_{j \in S} f(j)$. By $\delta_{1,g}$ we denote the Kronecker delta:
\[
\delta_{1,g} = \begin{cases} 1 & \textrm{ when } g=1, \\ 0 & \textrm{otherwise.}\end{cases}
\]

We work over a fixed algebraically closed field $k$ of characteristic $0$ throughout.

A  \emph{curve} over a field $F$ is a $\spec(F)$-scheme $C/\spec(F)$ that is proper over $\spec(F)$, geometrically connected, and pure of dimension $1$. A curve $C/\spec(F)$ over $F$ is a \emph{nodal curve} if $C$ is geometrically reduced and the completed local ring of $C \times_{\spec(F)} \spec(\overline{F})$ at a non-regular point is isomorphic to $\overline{F}[[x,y]]/(xy)$.  Here $\overline{F}$ is an algebraic closure of $F$.

A \emph{family of curves} over a $k$-scheme $T$ is a proper, flat morphism $C \to T$ whose fibers are curves. A family of curves $C \to T$ is a \emph{family of nodal curves} if the fibers are nodal curves.

A \emph{family of rank~$1$ torsion-free sheaves} over a family of curves $C \to T$ is a rank~$1$ sheaf $F$ on $C$, flat over $T$, whose fibers over the geometric points are torsion-free.

If $F$ is a rank~$1$ torsion-free sheaf on a nodal curve $C$ with irreducible components $C_i$, we define the \emph{multidegree} of $F$ by ${\deg}(F) := (\deg(F_{C_{i}}))$.  Here $F_{C_{i}}$ is the maximal torsion-free quotient of $F \otimes \mathcal{O}_{C_{i}}$. We define the \emph{(total) degree} of $F$ to be $\deg_C(F):=\chi(F)-1+p_a(C)$ where $p_a(C)= h^1(C, \mathcal{O}_C)$ is the arithmetic genus of $C$. The total degree and the multidegree of $F$ are related by the formula $\deg_C(F) = \sum \deg_{C_i} F - \delta_C(F)$, where $\delta_C(F)$ denotes the number of nodes of $C$ where $F$ fails to be locally free.

%{\bf Warning:} To keep the notation uniform, in this paper it will be convenient to define a \emph{hyperplane} to be the zero locus of a (not necessarily non-degenerate) affine linear functional on a vector space $V$. In other words, the empty set and the vector space itself will be called hyperplanes of $V$.

\section{Background}

\label{background}

\subsection{Graphs} \label{graphs} A graph $\Gamma$ is a tuple $(\operatorname{Vert}, \operatorname{HalfEdge}, \operatorname{a}, \operatorname{i})$ consisting of a finite set of vertices $\operatorname{Vert}$, a finite set of half-edges $\operatorname{HalfEdge}$, an assignment function $\operatorname{a} \colon \operatorname{HalfEdge} \to \operatorname{Vert}$, and a fixed point free involution $\operatorname{i} \colon \operatorname{HalfEdge} \to \operatorname{HalfEdge}$. The edge set  is defined as the quotient set $\operatorname{Edge}:= \operatorname{HalfEdge} / \operatorname{i}$. The endpoint of a half-edge $h \in \operatorname{Edge}$ is defined to be $v=a(h)$. A loop based at $v$ is an edge whose two endpoints coincide.

A $n$-marked graph is a graph $\Gamma$ together with a (genus) map $g \colon \operatorname{Vert}(\Gamma) \to \mathbb{N}$ and a (markings) map $p \colon \{ 1, \ldots, n \} \to \operatorname{Vert}(\Gamma)$.  We call $g(v)$ the genus of $v \in \operatorname{Vert}(\Gamma)$.  If $v=p(j)$, then we say that the $j$-th marking lies on the vertex $v$.

 A subgraph $\Gamma'$ of $\Gamma$ is always assumed to be proper ($\operatorname{Vert}(\Gamma') \subsetneq \operatorname{Vert}(\Gamma)$) and complete (if $h \in \operatorname{HalfEdge}(\Gamma)$ satisfies $a(h), a(i(h)) \in \operatorname{Vert}(\Gamma')$, then $h, i(h) \in \operatorname{HalfEdge}(\Gamma')$).  A subgraph of a $n$-marked graph is tacitly assumed to be given the induced genus and marking maps.

We say that a $n$-marked graph $\Gamma$ is \emph{stable} if it is connected (in the obvious sense, a bit tedious to write down), and if for all $v$ with $g(v)=0$, the sum of the number of half-edges with $v$ as an endpoint plus the number of markings lying on $v$ is at least $3$. The (arithmetic) \emph{genus} of $\Gamma$ is $g(\Gamma) := \sum_{v \in \operatorname{Vert}(\Gamma)} g(v) -\# \operatorname{Vert}(\Gamma)+  \# \operatorname{Edge}(\Gamma)+1 $.

An \emph{isomorphism} of $\Gamma= (\operatorname{Vert}, \operatorname{HalfEdge}, \operatorname{a}, \operatorname{i})$ to  $\Gamma'= (\operatorname{Vert}', \operatorname{HalfEdge}', \operatorname{a}', \operatorname{i}')$ is a pair of bijections $\alpha_V \colon \operatorname{Vert} \to \operatorname{Vert}'$ and $\alpha_{\operatorname{HE}} \colon \operatorname{HalfEdge} \to \operatorname{HalfEdge}'$ that satisfy the compatibilities $\alpha_{\operatorname{HE}} \circ i = i'$ and $\alpha_V \circ a = a'$. If $\Gamma$ and $\Gamma'$ are endowed with structures of $n$-marked graphs by the maps $(g,p)$ and by $(g',p')$ respectively, $(\alpha_V, \alpha_{\operatorname{HE}})$ is an isomorphism of $n$-marked graphs if it also satisfies the compatibilities $\alpha_V \circ p = p'$ and $\alpha_V \circ g = g'$. An \emph{automorphism} is an isomorphism of a graph to itself.

We fix once and for all a finite set ${G}_{g,n}$ of stable $n$-marked  graphs of genus $g$, one for each isomorphism class.

\subsection{Contractions}
 \label{contractions}
 We will need a notion  for contractions of stable graphs. This notion is ubiquitous in the literature on moduli of curves (see for example \cite[Appendix]{graberpanda}, where contractions are key to giving an algorithmic description of the intersection product of tautological classes). Here we first introduce the notion of a strict contraction and then define a contraction to be a strict contraction followed by an isomorphism. Unlike in \cite{graberpanda} and in other sources, our contractions contract precisely $1$ edge. % and isomorphisms of graphs are not a particular case of a contraction.

If $\Gamma$ is a $n$-marked graph and $e \in \operatorname{Edge}(\Gamma)$ is an edge,  the \emph{strict contraction} of $e$ in $\Gamma$ is the graph $\Gamma_e$ obtained from $\Gamma$ by \begin{enumerate} \item removing the half-edges of $\Gamma$ corresponding to $e$, \item replacing the two (possibly coinciding) endpoints $v_1$ and $v_2$ of $e$ by a unique vertex $v_e$, \item extending the marking function $p$ of $\Gamma$ to $v_e$ by $p_e(j):=v_e$ whenever $p(j)$ equals $v_1$ or $v_2$ and \item extending the genus function $g$ of $\Gamma$ to $v_e$ by \[g_e(v_e):= \begin{cases}g(v_1)+g(v_2) & \textrm{ when } e \textrm{ is not a loop},\\ g(v_1)+1 & \textrm{ when } e \textrm{ is a loop.} \end{cases}\]\end{enumerate}
If $\Gamma$ and $\Gamma'$ are $n$-marked graphs, a \emph{contraction} $c\colon \Gamma \to \Gamma'$ is the choice of an edge $e$ of $\Gamma$, and of an isomorphism of $\Gamma_e$ (the strict contraction of $e$ in $\Gamma$) to $\Gamma'$. The contraction $c$ is completely determined by the two maps it induces $c_V \colon \operatorname{Vert}(\Gamma) \to \operatorname{Vert}(\Gamma')$ (on vertices) and $c_{\operatorname{HE}} \colon \operatorname{HalfEdge}(\Gamma) \to \operatorname{HalfEdge}(\Gamma')$ (on half-edges).

\subsection{Moduli of curves} \label{notationmoduli}

In this paper we always assume that $g,n$ are natural numbers satisfying  $2g-2+n>0$. Under this assumption,  the moduli stack  $\Mb_{g,n}$  parameterizing families of stable $n$-pointed curves of arithmetic genus $g$ is a $k$-smooth and proper Deligne--Mumford stack. We will denote by $\pi \colon \Cb_{g,n} \to \Mb_{g,n}$ the universal curve, and by $\omega_{\pi}$ its relative dualizing sheaf.

	If $(C, p_1, \ldots, p_n)$ is a stable pointed curve, we define its \emph{dual graph} $\Gamma_{C}$ to be the $n$-marked graph whose vertices are the irreducible components of $C$,  whose edges are the nodes of $C$, whose genus map is given by assigning the geometric genus to each vertex, and whose markings map is the assignment $p \colon \{ 1, \ldots, n \} \to \operatorname{Vert}(\Gamma_{C})$ such that  $p(j)$ is the vertex containing $p_j$.

For each $\Gamma \in {G}_{g,n}$, the locus $\mathcal{M}_{\Gamma} \subseteq \Mb_{g,n}$ of stable curves whose dual graph is isomorphic to $\Gamma$ is locally closed. We are now going to fix a notation for some special stable graphs $\Gamma$ (and their corresponding loci $\mathcal{M}_{\Gamma}$), which will play an important role in this paper.

 For all pairs $(i,S)$ with $0\leq i \leq g$ and $S \subseteq [n]$, such that if $i=0$ then $|S| \geq 2$, and if $i=g$ then $|S| \leq n-2$, we define $\Gamma(i,S)$ to be the graph with $2$ vertices of genera $i$ and $g-i$ joined by $1$ edge with markings $S$ and $S^c$ respectively. The closure of the locus $\mathcal{M}_{\Gamma(i,S)}$ in $\Mb_{g,n}$ is a divisor that we we will denote by $\Delta(i,S)$. In this paper we will assume (in summation formulas etc.) that the set of indices $\{(i,S)\}$ for $0 \leq i \leq g$ and $S \subseteq [n]$ satisfies the additional requirement that
 \begin{enumerate}
  \item if $n=0$, then $i < g-i$,
 \item if $n \geq 1$, then $1 \in S$.
 \end{enumerate}
We adopt this convention so that there is a bijection between the set of indices $\{(i,S)\}$ and
the set of boundary divisors $\Delta(i,S)\subsetneq \Mb_{g,n}$ whose inverse image $\pi^{-1}(\Delta(i,S))$ in the universal curve
$\pi \colon \Cb_{g,n}\to \Mb_{g,n}$ consists of $2$ irreducible components.

When $g \geq 1$ and $n \geq 1+ \delta_{1,g}$, for each $j = 1+ \delta_{1,g}, \ldots, n$ we denote by $\Gamma_j$ the graph with $2$ vertices of genera $0$ and $g-1$ respectively, joined by $2$ edges, and with marking $j$ on the first vertex and all other markings on the second vertex.

Another collection of curves that will play a crucial role in this paper, and that includes those discussed in the previous two paragraphs, consists of the so-called \emph{vine curves}. These are by definition curves with $2$ smooth irreducible components or, equivalently, curves whose dual graph has $2$ vertices and no loops.

\subsection{The relative Picard group of the universal curve}
We will often need to work with the relative Picard group of the universal curve  $\pi \colon \Cb_{g,n} \to \Mb_{g,n}$, and with its affine subspaces of elements of fixed fiberwise degree. For this reason, we introduce the following definition/notation.
\begin{definition}
We denote by
\[ \picrel_{g,n}(\mathbb{Z}): = \Pic(\Cb_{g,n}) / \pi^* ( \Pic( \Mb_{g,n})) \]
the  relative Picard group of the universal curve $\pi$ and by
\[\picrel_{g,n}(\mathbb{R}): = \picrel_{g,n}(\mathbb{Z}) \otimes_{\mathbb{Z}} \mathbb{R}\]
the relative Picard group of real line bundles.

For every $d \in \mathbb{Z}$ (resp.~$\in \mathbb{R}$), we let $\picrel^d_{g,n}(\mathbb{Z})$ (resp.~$\picrel_{g,n}^d(\mathbb{R})$) be the affine subspace of $\picrel_{g,n}(\mathbb{Z})$ (resp.~of $\picrel_{g,n}(\mathbb{R})$) of elements of fiberwise degree $d$.
\end{definition}

Let $\Sigma_j$ be the $j$-th section of the universal curve $\pi \colon \Cb_{g,n} \to \Mb_{g,n}$, and $\omega_{\pi}$ be the relative dualizing sheaf.

We make the following (canonical) choice for a base point in $\picrel^d_{g,n}(\mathbb{R})$: \begin{equation} \label{basepoint1} \begin{cases} \frac{d}{2g-2} \cdot \omega_{\pi} & \textrm{ when } g \geq 2, \\ d \cdot \Sigma_1 & \textrm{ when } g\leq 1. \end{cases} \end{equation}
The choice of a base point makes $\picrel_{g,n}^d(\mathbb{R})$ into a vector space isomorphic to $\picrel^0_{g,n}(\mathbb{R})$.

We now recall what will later be needed about the structure of the free abelian group $\picrel_{g,n}(\mathbb{Z})$ and the structure of its subgroup $\picrel_{g,n}^0(\mathbb{Z})$. The results we state essentially follow from the description of the Picard group of $\Mb_{g,n}$ that was originally given by Arbarello--Cornalba in \cite{acpicard}.

\begin{definition} \label{generators} For each pair $(i,S)$ satisfying the assumptions of Section~\ref{notationmoduli}, we define ${C}_{i,S}^{+}$ and ${C}_{i,S}^{-}$ to be the $2$ components of the universal curve $\pi \colon \Cb_{g,n} \to \Mb_{g,n}$ over the boundary divisor $\Delta(i,S)$. The component $C_{i,S}^+$ is the one that contains the first marked point, and, when $n=0$, it is the component of lowest genus. We define $W_{g,n}$ to be the subgroup of $\picrel_{g,n}^0(\mathbb{Z})$ generated by the line bundles $\mathcal{O}(C_{i,S}^+)$.

For $j=1+ \delta_{1,g}, \ldots, n$, define the \emph{twisted sections} \[ T_j:=\begin{cases}\mathcal{O} (\Sigma_j - \Sigma_1) & \textrm{ if } g=1, \\ \mathcal{O}((2g-2)\Sigma_j) \otimes \omega_{\pi}^{\otimes -1} & \textrm{ if } g \geq 2. \end{cases}\]  (When $g=0$ we have intentionally defined no twisted sections).
\end{definition}

\begin{fact} \label{arbarellocornalba} The group  $\picrel_{g,n}(\mathbb{Z})$ is \emph{freely} generated \begin{enumerate} \item by all components $\mathcal{O}(C^+_{i,S})$ over the boundary divisors and by one section $\Sigma_i$ when $g=0$, \item by all components $\mathcal{O}(C^+_{i,S})$ and by all sections  $\Sigma_1, \ldots, \Sigma_n$ when $g=1$, \item by all components $\mathcal{O}(C^+_{i,S})$, by all sections $\Sigma_1, \ldots, \Sigma_n$ and by the relative dualizing sheaf $\omega_{\pi}$ when $g \geq 2$.\end{enumerate} \end{fact}

This result was stated incorrectly for $g=0,1$ in a previous version of this paper. Many thanks to Filippo Viviani for catching the error.

\begin{proof}	
	Identify the universal curve $\pi \colon \Cb_{g,n} \to \Mb_{g,n}$ with the map forgetting the last point and stabilizing $\pi \colon \Mb_{g,n+1} \to \Mb_{g,n}$. As observed in \cite{acg2}, the Picard group of each moduli stack of stable pointed curves is free, and rational and homological equivalence coincide for codimension $1$ cycles.

By \cite{acg2}, integral generators of the Picard group of $\Mb_{g,n+1}$ are $\lambda:= c_1(\pi_* (\omega_{\pi}))$ together with the $\psi$-classes $\psi_i=\sigma_i^*(\omega_{\pi})$ for $i=1,\ldots, n+1$, the boundary divisors $\Delta(i,S)$ introduced in Section~\ref{notationmoduli}, and the boundary divisor $\Delta_{\text{irr}}$ that generically parameterizes irreducible singular curves. These generators are free when $g \geq 3$ but there are relations for $g=0,1$ and $2$.

Because $\lambda$ is the pullback via $\pi$ of the analogue class from $\Mb_{g,n}$, it is zero in the quotient group. The same happens with $\Delta_{\text{irr}}$.  Another element in the Picard group (that will appear later in this proof via \cite[Lemma~1.2 and Theorem~2.2.b]{ac}) is \[\kappa_1 = \pi_*(c_1 (\omega_{\pi}(\Sigma_1 + \ldots + \Sigma_n))^2).\] This class can be expressed in terms of the previous basis by means of Mumford's relation:
	\[
	12 \lambda = \kappa_1 + \delta - \psi
	\]
	(here $\delta$ is the sum of all boundary divisors and $\psi$ is the sum of all $\psi_i$'s). In the relative Picard group we therefore have that $\kappa_1$ equals $\psi-\delta$.

	We then eliminate more generators by using the relations indicated in \cite[Lemma~1.2 and Theorem~2.2]{ac}.  Under the identification of $\Cb_{g,n}$ with $\Mb_{g,n+1}$ each section $\Sigma_j$ corresponds to the boundary divisor $\Delta(0, \{j, n+1\})$, and the relative dualizing sheaf $\omega_{\pi}$ (or rather its first Chern class) corresponds to \[\psi_{n+1} - \Delta(0, \{1, n+1\}) - \ldots - \Delta(0, \{n, n+1\}).\] Then use the right-hand side of the equalities in \cite[Lemma 1.2]{ac} to eliminate redundant generators. This gives that the components over the boundary divisors, the sections, and the relative dualizing sheaf are integral generators of the relative Picard group. When $g \geq 3$ by \cite[Theorem~2.2]{ac} there is no other relation, and this concludes our proof.
	
When $g \leq 2$ there are more relations indicated in \cite[Theorem~2.2]{ac}. 	When $g~=~2$, the extra relation given in \cite[Theorem~2.2.(b)]{ac} is redundant. When $g\leq 1$ use the relations indicated in (c) and (d) of \cite[Theorem~2.2]{ac} to eliminate the relative dualizing sheaf when $g=1$, and all but one section when $g=0$.
\end{proof}  %In particular, the collection $\{\mathcal{O}(C^+_{i,S}) \}$ is a \emph{free} set of generators for $W_{g,n}$.

By singling out the degree zero elements, we deduce the following corollary.
\begin{corollary} \label{piczero} The vector space  $\picrel_{g,n}^0(\mathbb{R})$ is \emph{freely} generated by the components $\mathcal{O}(C^+_{i,S})$ and, when $g \geq 1$, by the twisted sections $T_j$. % for $j= 1+ \delta_{1,g}, \ldots, n$.
\end{corollary}
(In a previous version of this paper we incorrectly stated that this result also holds integrally).

In particular, when either $g=0$ or $n=0$, the vector space $\picrel_{g,n}^0(\mathbb{R})$ coincides with $W_{g,n} \otimes_{\mathbb{Z}} \mathbb{R}$. %When $g \geq 1$ we have

\label{Picard}

\section{The universal stability space}  \label{Section: StabilitySpace}

In this section we construct and study the stability $\mathbb{R}$-vector space $V_{g,n}$, whose affine subspaces $V_{g,n}^d$ of elements of total degree $d \in \mathbb{Z}$ are the stability spaces of $\phi$-compactified universal Jacobians over $\Mb_{g,n}$, which we will construct in Section~\ref{Section: construction}.

In \cite[Section 3.1]{kasspa} we introduced a similar stability space, which we called $V_{g,n}^{\text{TL}}$, the stability space of degree $\phi$-compactified universal Jacobians of degree $g-1$ over moduli of treelike curves. For more details on $V_{g,n}^{\text{TL}}$ and its relation to the space $V_{g,n}^{g-1}$ we introduce here, we direct the reader to Remarks~\ref{compareoldpaper} and \ref{compareoldpaper2}.

Our main result is Theorem~\ref{pic=stab}, which describes the stability space $V_{g,n}$ as the real relative Picard group of the universal curve $\pi \colon \Cb_{g,n} \to \Mb_{g,n}$. An important by-product is Corollary~\ref{isoalphan}, where we show that, for fixed $d$, every degree $d$ stability parameter is uniquely determined by its restriction to all stable curves with $2$ smooth irreducible components that (1) have precisely $1$ separating node, or (2) have $2$ nodes and $1$ component of genus $0$ that carries a unique marked point (the curves whose dual graph is $\Gamma(i,S)$ or $\Gamma_j$ respectively, see Section~\ref{notationmoduli}). In order to make the statement of our main result more transparent, in this section we allow $d$ to be a real number.

\begin{definition} Given $\Gamma \in {G}_{g,n}$ a stable $n$-marked graph of genus $g$, we denote by $V(\Gamma):=\mathbb{R}^{\operatorname{Vert}(\Gamma)}$ the free $\mathbb{R}$-vector space generated by the vertices of $\Gamma$.  For $d \in \mathbb{R}$, the affine subspace $V^d(\Gamma)$ is the set of $\phi \in V(\Gamma)$ such that $\sum_{v \in \operatorname{Vert}(\Gamma)} \phi(v) = d$. \end{definition}

Every automorphism $\alpha$ of $\Gamma$ induces an automorphism of $V(\Gamma)$ defined by $\alpha(\phi)(v)= \phi(\alpha(v))$. An element $\phi \in V(\Gamma)$ is automorphism invariant if $\phi(v) = \phi(\alpha(v))$ for all $v \in \operatorname{Vert}(\Gamma)$. A vector $\phi \in \Pi_{\Gamma \in {G}_{g,n}} V(\Gamma)$ is \emph{automorphism invariant} if for every $\Gamma \in {G}_{g,n}$, the component $\phi(\Gamma)$ of $\phi$ along $\Gamma$ is automorphism invariant in the sense just defined.

Suppose that $c \colon \Gamma_1 \to \Gamma_2$ is a contraction of stable marked  graphs as defined in Section~\ref{contractions}.  We say that $\phi(\Gamma_1) \in V(\Gamma_1)$ is $c$-compatible with $\phi(\Gamma_2) \in V(\Gamma_2)$ if
	\begin{equation} \label{Eqn: CompatibilityRelation}
		\phi(\Gamma_2)(v_2) = \sum \limits_{c(v_1)=v_2} \phi(\Gamma_1)(v_1)
	\end{equation}
	for all vertices $v_2 \in \operatorname{Vert}(\Gamma_2)$.  An element $\phi \in \Pi_{\Gamma \in {G}_{g,n}} V(\Gamma)$ is \emph{compatible with contractions} if its components are $c$-compatible for every contraction $c \colon \Gamma_1 \to \Gamma_2$.

\begin{definition} \label{stabilityspacedef} We define $V_{g,n}$ to be the subspace of $\prod_{\Gamma \in {G}_{g,n}} V(\Gamma)$ of vectors that are automorphism invariant and compatible with contractions. For $d \in \mathbb{R}$, we define $V_{g,n}^d$ to be the affine subspace of  vectors $\phi \in V_{g,n}$ that satisfy $\sum_{v \in \operatorname{Vert}(\Gamma)} \phi({\Gamma})(v) = d$ for all $\Gamma \in {G}_{g,n}$.
\end{definition}

  \begin{remark}   We could have equivalently defined $V_{g,n}^d$ as the subspace of vectors  $\prod_{\Gamma \in {G}_{g,n}} V^d(\Gamma)$ that are automorphism invariant and compatible with contractions.\end{remark}

If $\pi \colon \Cb_{g,n} \to \Mb_{g,n}$ denotes the universal curve, there is a natural multidegree homomorphism $\deg \colon \Pic(\Cb_{g,n}) \to V_{g,n}$ defined by associating to $L$ the vector $\phi=\deg(L)$ whose $\Gamma$-component $\phi(\Gamma)$ is the multidegree of $L$ on any stable pointed curve whose dual graph is isomorphic to $\Gamma$.

There is a natural choice of a basepoint in $V_{g,n}^d$ that mirrors the basepoint we chose in Section~\ref{Picard} for the relative Picard group.

\begin{definition} \label{basepoint} We define the  \emph{canonical parameter} as follows \begin{equation} \label{phican} \phi_{\text{can}}^d:= \begin{cases} \frac{d}{2g-2} \cdot \deg(\omega_{\pi}) & \textrm{ when }g \geq 2, \\ d \cdot \deg(\Sigma_1) & \textrm{ when } g \leq 1.\end{cases} \end{equation}
\end{definition}

In order to state the main result of this section, we first observe that for every $L \in \Pic(\Mb_{g,n})$, the stability parameter $\deg(\pi^*(L))\in V_{g,n}$ is trivial, so the multidegree map descends to a well-defined map $\deg \colon \picrel_{g,n}(\mathbb{Z}) \to V_{g,n}$.
\begin{theorem} \label{pic=stab} The multidegree homomorphism $\deg$ induces an isomorphism
\[
\deg \colon \picrel_{g,n}(\mathbb{R})=\Pic(\C_{g,n}) / \pi^* ( \Pic( \Mb_{g,n})) \otimes_{\mathbb{Z}} \mathbb{R} \to V_{g,n}
\]
from the relative Picard group of real line bundles to the stability space.
\end{theorem}

Before we prove Theorem~\ref{pic=stab}, we now give another  description of the stability spaces that will be the one we will mostly use in this paper.  To this end, we make the following definition.

\begin{definition} \label{CDT} Let
\[
		T_{g,n}:=   \bigoplus_{\substack{ \# \operatorname{Vert}(\Gamma)=2\\ \Gamma \textrm{ has no loops}}} V^0(\Gamma).
	\]
Then define:

\begin{enumerate}

\item The vector space $C_{g,n}$ is the quotient space of $T_{g,n}$ obtained as the direct sum of all $V^0(\Gamma(i,S))$ (see Section~\ref{notationmoduli} for the definition of $\Gamma(i,S)$).

\item The vector space $D_{g,n}$ is the quotient space of $T_{g,n}$ obtained as the direct sum of all $V^0(\Gamma_j)$ for $j=1+ \delta_{1,g},\ldots, n$ (see Section~\ref{notationmoduli} for the definition of $\Gamma_j$. In particular, $D_{0,n} = \{0\}$).
\end{enumerate}
There are natural projections $p_C \colon T_{g,n} \to C_{g,n}$ and $p_D \colon T_{g,n} \to D_{g,n}$.
\end{definition}
There is a natural restriction map $\rho \colon V_{g,n}^0 \to T_{g,n}$. More generally, for every  $d\in \mathbb{R}$, there is a natural map $\rho^d \colon V_{g,n}^d \to T_{g,n}$ obtained by composing $\rho$ with the translation $\phi \mapsto \phi - \phi_{\text{can}}^d$.  Choosing the canonical parameter $\phi_{\text{can}}^d$ for the origin in $V_{g,n}^d$ makes $\rho^d$ into a homomorphism of vector spaces. We have then the following alternative description of each degree $d$ stability space.
 \begin{corollary} \label{isoalphan}
The composite homomorphism  \[(p_C \oplus p_D) \circ \rho^d \colon V_{g,n}^d \to C_{g,n} \oplus D_{g,n}\] is an isomorphism. \end{corollary}

We now aim to prove Theorem~\ref{pic=stab} and Corollary~\ref{isoalphan}. Here is the idea of our proof when $g \geq 2$. It is not hard to reduce both results to proving that, in degree zero, both maps
\[
\picrel_{g,n}^0(\mathbb{R}) \to V_{g,n}^0 \to C_{g,n} \oplus D_{g,n}
\]
are isomorphisms. Injectivity of $\picrel_{g,n}^0(\mathbb{R}) \to C_{g,n} \oplus D_{g,n}$ follows by computing the bidegree of the free generators of the Picard group on curves whose dual graph is $\Gamma(i,S)$ and $\Gamma_j$ (Lemma~\ref{degreecalculation}), and observing that the resulting matrix is nonsingular. From this we immediately deduce that  $\picrel_{g,n}^0(\mathbb{R}) \to C_{g,n} \oplus D_{g,n}$ is an isomorphism, because the source and the target have the same dimension. Our results follow if we can prove that $\picrel_{g,n}^0(\mathbb{R}) \to V_{g,n}^0$ is surjective, or equivalently that $V_{g,n}^0 \to C_{g,n} \oplus D_{g,n}$ is injective. This is the content of Proposition~\ref{Prop: pic=stab}, which we prove inductively in $n$. An important intermediate step is Lemma~\ref{Lemma: PhiDeterminedByTwoComponents}, where we show that $V_{g,n}^0 \to T_{g,n}$ is also injective. The base case of the induction $n=0$ is settled by combining Lemma~\ref{Lemma: PhiDeterminedByTwoComponents} and Lemma~\ref{LemmaTLsurjective}.

We now compute the bidegree of the free generators we chose in Corollary~\ref{piczero} for the relative Picard group $\picrel_{g,n}^0(\mathbb{R})$ on the particular vine curves that appear in parts (1) and (2) of Definition~\ref{CDT}. When ordering the $2$ components of the curve, we follow the same convention that we chose in Definition~\ref{generators} to order the $2$ components $C_{i,S}^+$ and $C_{i,S}^-$ of the inverse image of $\Delta_{i,S}$ in the universal curve.

\begin{lemma} \label{degreecalculation} For $g \geq 2$, the bidegrees of the components $\mathcal{O}({C}_{i',S'}^{+})$ and of the twisted sections $T_k$ on curves whose dual graph is $\Gamma(i,S)$ and $\Gamma_j$ is given by the following formulas:

\begin{equation}\label{one}
\deg\left(\mathcal{O}(C_{i,S}^{+}) |\Gamma(i',S')\right) = \begin{cases}
										(- 1, + 1) 	& \text{ if } (i',S')=(i, S),\\
										(0, 0)		& \text{ if } (i',S') \ne (i,S).
									\end{cases} 	\end{equation}

\begin{equation} \label{two} \deg\left(\mathcal{O}(C_{i,S}^{+}) |\Gamma_j \right)  = (0,0) 	\end{equation}

\begin{equation} \label{three}
\deg \left(T_k |\Gamma(i,S)\right) = \begin{cases}
										(2g-2i-1, 2i+1-2g) 	& \text{ if } k \in S,\\
										(1-2i, 2i-1)		& \text{ if } k \notin S;
									\end{cases} 	\end{equation}

 \begin{equation} \label{four} {\deg}\left( T_k |\Gamma_j\right) = \begin{cases}
										(2g-2  ,  2-2g) 	& \text{ if } j=k,\\
										(0, 0)		& \text{ if } j \ne k.
									\end{cases} 	\end{equation}
\end{lemma}

\begin{proof} Straightforward.
\end{proof}
Combining Lemma~\ref{degreecalculation} with Corollary~\ref{piczero} we deduce that rational and numerical equivalence are equivalent in $\picrel_{g,n}^0(\mathbb{R})$, and from this we deduce that the composite map $(p_C \oplus p_D)\circ \rho \circ \deg\colon \picrel_{g,n}^0(\mathbb{R}) \to C_{g,n} \oplus D_{g,n}$ is injective.

An important fact  that we will use throughout  is that every element $\phi \in V_{g,n}^0$ is completely determined by its value over all vine curves (i.e.~curves with $2$ smooth irreducible components, see Section~\ref{notationmoduli} and Definition~\ref{CDT}).

\begin{lemma} \label{Lemma: PhiDeterminedByTwoComponents}
	The restriction
	$
		\rho \colon V_{g,n}^0 \to   T_{g,n}
	$
	 is injective.
\end{lemma}

\begin{proof}
 Let $\Gamma'\in {G}_{g,n}$ be a stable graph. By applying compatibility with contractions we can assume without loss of generality that $\Gamma'$ has no loops. Then consider a spanning tree $\Gamma$ of $\Gamma'$, and run the injectivity part of the proof in \cite[Lemma 3.9]{kasspa}, with the only difference being that the right-hand side of \cite[Equation (15)]{kasspa} should equal zero.
\end{proof}

\begin{lemma} \label{LemmaTLsurjective} Let $\phi \in V_{g,0}^0$ and $\Gamma \in {G}_{g,0}$ be a graph with $2$ vertices joined by $\geq 2$ edges. Then $\phi(\Gamma)$ is trivial.
\end{lemma}

\begin{proof}
We begin by fixing the notation for loopless graphs with $2$ vertices and at least $2$ edges, when $n=0$. Let $\alpha, i, j \in \mathbb{N}$ such that $\alpha+i+ j -1= g$, $\alpha \geq 2$ and subject to the stability condition $\min(i, j) = 0 \implies \alpha \geq 3$. Define the stable graph $\Gamma(\alpha, i, j) \in {G}_{g,0}$ to consist of $2$ vertices $v_1, v_2$ of genera $i$ and $j$ respectively, with $\alpha$ edges connecting $v_1$ to $v_2$. We aim to prove that $\phi(\Gamma(\alpha,i,j))=(0,0)$.

We first prove our claim in the special case when $j=0$ (so $3 \leq \alpha = g+1- i$). Consider the trivalent graph $\GSym_g$ that has $2g-2$ vertices $v_1, \ldots, v_{2g-2}$ of genus $0$, where each vertex $v_i$ is joined to $v_{i-1}$, $v_{i+1}$  and $v_{i+g-1}$ (here indices should be considered modulo $2g-2$). The cyclic group of order $2g-2$ acts on $\GSym_g$ and its induced action on the set of vertices $\operatorname{Vert}(\GSym_g)$ is transitive. By automorphism invariance, the component $\phi(\GSym_g)$ of $\phi \in V_{g,0}^0$ along $\GSym_g$ is trivial.

  Choose $2g-\alpha$ consecutive vertices of $\GSym_g$ and contract them to $1$ vertex, possibly contracting all loops based at that vertex in the process. Then  also contract the complement set of $\alpha-2$ vertices of $\GSym_g$ to a second vertex. The resulting graph is isomorphic to $\Gamma(\alpha, i, 0)$ and by compatibility with contractions we have that the $\Gamma(\alpha, i, 0)$-component of $\phi \in V_{g,0}^0$ is trivial, thus proving the claim in the case $j=0$.

We now deduce that the $\Gamma(\alpha, i, j)$-component of $\phi$ is zero for all $\alpha, i, j$ by using contractions to relate $\Gamma(\alpha, i, j)$ to a graph of the form $\Gamma(t_1, t_2, 0)$, specifically the graph $\Gamma(i-j+2, \alpha+2j-2,0)$.

Assume without loss of generality that $i>j$. Consider the graph with 4 vertices  $w_1, w_2, w_3, w_4$ where $w_1$ and $w_2$ have genus $j$, and $w_3$ and $w_4$ have genus $0$. The vertices $w_1$ and $w_2$ are joined by $\alpha-1$ edges; $w_1$ is joined to $w_3$ by $1$ edge, and so is $w_2$ to $w_4$. Finally, $w_3$ and $w_4$ are joined by $i-j+1$ edges. By invariance under automorphism, the component of $\phi \in V_{g,0}^0$ along this graph equals $(a, a, -a, -a)$ for some $a \in \mathbb{R}$.

Contracting the vertices $w_1, w_2$ and $w_4$ to a single vertex and then contracting all remaining loops on it, produces a graph isomorphic to $\Gamma(i - j +2,  \alpha+2j-2, 0)$. Because we have already computed that the component of $\phi$ along this graph is trivial, we deduce that $a=0$.

Contracting the vertices $w_1, w_3$ and $w_4$ to a single vertex and then contracting all remaining loops on it, produces a graph isomorphic to $\Gamma(\alpha, i, j)$, so by compatibility with contractions the component $\phi(\Gamma(\alpha, i, j))$ is also trivial, and the statement is proven.
\end{proof}

The following is the key part of the proof of the main result of this section.

\begin{proposition} \label{Prop: pic=stab}
 Both the multidegree homomorphism \begin{equation} \deg \colon \picrel_{g,n}^0(\mathbb{R}) \to V_{g,n}^0\end{equation} and the composition \begin{equation}  (p_C \oplus p_D) \circ \rho \colon V_{g,n}^0 \to C_{g,n} \oplus D_{g,n} \end{equation} are isomorphisms.
 \end{proposition}

\begin{proof}

In Lemma~\ref{degreecalculation} we computed the bidegrees of all generators of $\picrel_{g,n}^0(\mathbb{Z})$ given in Corollary~\ref{piczero} against curves whose dual graph is $\Gamma(i,S)$ and $\Gamma_j$. The corresponding square matrix is nonsingular: it consists of four blocks \[ \begin{pmatrix}\eqref{one}&\eqref{two} \\ \eqref{three} & \eqref{four} \end{pmatrix},\]  where \eqref{one} is the identity,  \eqref{two} is the zero matrix and \eqref{four} is $(2g-2+\delta_{1,g})$ times the identity. Combining this with Corollary~\ref{piczero}, we deduce that the composite map
\begin{equation} \label{stepzero}
(p_C \oplus p_D)\circ \rho \circ \deg \colon  \picrel_{g,n}^0(\mathbb{R}) \to C_{g,n} \oplus D_{g,n} \textrm{ is an isomorphism.}
\end{equation}
 Because of this, both claims of this proposition follow by proving that
 \begin{equation} (p_C \oplus p_D) \circ \rho \colon V_{g,n}^0 \to C_{g,n} \oplus D_{g,n} \textrm{ is injective.} \label{firststep} \end{equation}

The $g=0$ case is easily settled. We have that $T_{g,n} = C_{g,n}$, so \eqref{firststep} follows immediately from Lemma~\ref{Lemma: PhiDeterminedByTwoComponents}.

From now on we assume $g \geq 1$ and we aim for proving \eqref{firststep}. We simplify the problem by quotienting out the images of the space $W_{g,n} \otimes_{\mathbb{Z}} \mathbb{R}$ via $\deg$ and via $ (p_C \oplus p_D)\circ \rho \circ \deg$. Applying parts \eqref{one} and \eqref{two} of Lemma~\ref{degreecalculation}, we deduce that the image via $(p_C \oplus p_D)\circ \rho \circ \deg$ of $W_{g,n} \otimes \mathbb{R}$ equals $C_{g,n} \oplus \{0 \}$. Call $\Ker_{g,n}$ the kernel of $V_{g,n}^0 \to C_{g,n}$.  By \eqref{stepzero} we know that $p_D \circ \rho \colon \Ker_{g,n} \to D_{g,n}$ is surjective. For these reasons, to prove \eqref{firststep} it is enough to prove the inequality
\begin{equation} \label{secondstep}
\dim(\Ker_{g,n}) \leq n - \delta_{1,g}.
\end{equation}
(A posteriori, \eqref{secondstep} will be an equality.) We will prove Inequality~\eqref{secondstep} inductively in $n$.

When $g=n=1$ it is straightforward to check that $V_{1,1}^0= \{0\}$. When $g \geq 2$, the $n=0$ case of \eqref{secondstep} follows from Lemma~\ref{LemmaTLsurjective}, which implies that $(p_C \oplus p_D) \circ \rho= p_C \circ \rho$ is injective.

From now on, we apply the induction hypothesis. Assuming \eqref{secondstep} holds, we aim to prove it holds for $n+1$, i.e.
\begin{equation} \label{thirdstep}
\dim(\Ker_{g,n+1}) \leq n+1 - \delta_{1,g}.
\end{equation}
Define the subspace $K_{g,n+1} \subseteq \Ker_{g,n+1}$ as the subspace of vectors $\phi\in \Ker_{g,n+1}$ such that $\phi(\Gamma)(v)$ equals zero for all $\Gamma \in {G}_{g,n}$ and $v \in \operatorname{Vert}(\Gamma)$ such that $v$ becomes unstable after forgetting the last marking $n+1$. To prove \eqref{thirdstep} it is enough to prove the two inequalities \begin{equation} \label{inequality3} \dim (\Ker_{g,n+1}) \leq \dim (K_{g,n+1}) +1, \textrm{ and } \dim(K_{g,n+1}) \leq \dim ( \Ker_{g,n}).\end{equation}

 We first prove the inequality $ \dim (\Ker_{g,n+1}) \leq \dim (K_{g,n+1}) +1$. By applying Lemma~\ref{Lemma: PhiDeterminedByTwoComponents} again, we identify $\Ker_{g,n+1}$  with a subspace of $T_{g,n+1}$, and observe that $K_{g,n+1}$ contains (a posteriori, it will coincide with) the codimension-$1$ subspace of $\Ker_{g,n+1}$  of vectors $\phi$ whose component $\phi(\Gamma_{n+1})$ is trivial (for $\Gamma_{n+1}$ defined in Section~\ref{notationmoduli}).

We  prove the inequality $\dim(K_{g,n+1}) \leq \dim ( \Ker_{g,n})$ by showing the existence of a surjective linear map $\lambda \colon\Ker_{g,n} \to K_{g,n+1}$, which we define as follows. Let $\Gamma \in {G}_{g,n+1}$ and $\Gamma' \in {G}_{g,n}$ be obtained from $\Gamma$ by forgetting the last marking and possibly by stabilizing. If $\phi \in \Ker_{g,n}$, then $\lambda(\phi)$ is defined to equal $\phi$ on all vertices of $\Gamma$ that correspond bijectively to vertices of $\Gamma'$, and $0$ on the extra vertex (if any).  Because $\phi$ is automorphism-invariant and compatible with contractions, so is $\lambda(\phi)$. Because $\phi \in \Ker_{g,n}$, and because of the very definition of $\lambda$, we have that $\lambda(\phi) \in K_{g,n+1}$. Again by its very definition, $\lambda \colon \Ker_{g,n} \to K_{g,n+1}$ is surjective (a posteriori, it will be an isomorphism). This concludes our proof.
\end{proof}

From Proposition~\ref{Prop: pic=stab} we easily deduce Theorem~\ref{pic=stab} and Corollary~\ref{isoalphan}.

\begin{proof} (Of Theorem \ref{pic=stab})
By Proposition~\ref{Prop: pic=stab} we have that the multidegree map \[\deg \colon \picrel_{g,n}^0 (\mathbb{R}) \to V_{g,n}^0\] is an isomorphism. Moreover, both $\picrel_{g,n}^0(\mathbb{R}) \subsetneq \picrel_{g,n}(\mathbb{R})$ and $V_{g,n}^0 \subsetneq V_{g,n}$ are inclusions of codimension-$1$ subspaces. By definition, the multidegree map $\deg$ maps the base point of $\picrel_{g,n}^d(\mathbb{R})$ we defined in Equation~\eqref{basepoint1} to the base point of $V_{g,n}^d$ we defined in Equation~\eqref{phican}. This concludes our proof.
\end{proof}

\begin{proof} (Of Corollary \ref{isoalphan})
By Proposition~\ref{Prop: pic=stab} we have that the projection \[(p_C \oplus p_D) \circ \rho \colon V_{g,n}^0 \to C_{g,n} \oplus D_{g,n}\] is an isomorphism. Translation by $\phi_{\text{can}}^d$ is also an isomorphism of vector spaces $V_{g,n}^d \to V_{g,n}^0$. This concludes our proof.
\end{proof}

We conclude the section with some remarks on our Theorem \ref{pic=stab}.

\begin{remark} \label{compareoldpaper} In \cite[Definition~3.7]{kasspa} we introduced a vector space $V_{g,n}^{\text{TL}}$ governing universal stability over moduli of treelike curves, and when $d=g-1$. (Treelike curves are stable pointed curves whose nodes are either separating or belong to a unique irreducible component). By definition, there is a quotient map $q \colon V_{g,n}^{g-1} \to V_{g,n}^{\text{TL}}$ and Corollary~\ref{isoalphan} gives an isomorphism of $V_{g,n}^{\text{TL}}$ with $C_{g,n}$. By Corollary~\ref{isoalphan} we also have that if $\psi \in V_{g,n}^{\text{TL}}$, the preimage $q^{-1}(\psi)$ can be identified with $D_{g,n}$. \end{remark}

\begin{remark} \label{optimalonct} By Lemma~\ref{Lemma: PhiDeterminedByTwoComponents} the restriction map $V_{g,n}^0 \to T_{g,n}$ is injective, and in Corollary~\ref{isoalphan} we have made a choice of a subset $Q$ of the set of $2$-vertices loopless graphs such that the restriction map $V_{g,n}^0 \to \prod_{\Gamma \in Q} V^0(\Gamma)$ is a surjection.

We claim that if $Q$ is any such subset of the set of $2$-vertices loopless graphs, then $Q$ must contain all graphs $\Gamma(i,S)$. Indeed as a consequence of Equation~\eqref{one} we have that $W_{g,n} \to V_{g,n}^0$ is injective. Moreover, similar to what was seen in Equation~\eqref{two}, the bidegree of all elements of $W_{g,n}$ on curves with $2$ smooth components and at least $2$ nodes is trivial. It follows that the restriction map $V_{g,n}^0 \to \prod_{\Gamma \in Q} V^0(\Gamma)$ would not be surjective, were $Q$ not to contain some graph $\Gamma(i,S)$.

In this sense, choosing all graphs $\Gamma(i,S)$ in Definition~\ref{CDT} is natural. On the other hand, choosing all graphs $\Gamma_j$ in loc.~cit.~is arbitrary --- one could have opted for another choice of $n- \delta_{1,g}$ loopless graphs with $2$ vertices and at least $2$ edges.
\end{remark}
\label{stabilityspacesec}

\section{Compactified universal Jacobians}

\label{Section: construction}

For all nondegenerate $\phi \in V_{g,n}^d$ we construct $\phi$-compactified universal Jacobians $\Jb_{g,n}(\phi)$ as $k$-smooth, proper Deligne--Mumford stacks that are flat over $\Mb_{g,n}$ (see Definition~\ref{Def: moduliStacks} and Corollary~\ref{Cor: JbExists}). It is essentially Theorem~\ref{pic=stab} that allows us to directly construct all such universal Jacobians  from Simpson's result \cite[Theorem 1.21]{simpson}. We start by reviewing the notion of (Oda--Seshadri) $\phi$-stability on a single curve. % that we introduced in \cite{kasspa} for $d=g-1$.

	Let $(C, p_1, \ldots, p_n)$ be a  stable pointed curve with dual graph $\Gamma$ and $C_0 \subseteq C$ be a subcurve (i.e.~the union of some of the irreducible components of $C$) with dual graph $\Gamma_0 \subseteq \Gamma$. We write  $\deg_{\Gamma_0}(F)$ for the total degree $\deg_{C_{0}}(F)$ of the maximal torsion-free quotient of $F \otimes \mathcal{O}_{C_{0}}$ and $C_{0} \cap C_{0}^{c}$ or $\Gamma_{0} \cap \Gamma_{0}^{c}$ for the set of edges $e \in \operatorname{Edge}(\Gamma)$ that join a vertex of $\Gamma_0$ to a vertex of its complement $\Gamma_{0}^{c}$. Given a rank~$1$ torsion-free sheaf $F$ of degree $d$, we have $\deg_{C_{0}}(F)+\deg_{C_{0}^{c}}(F)=d-\delta_{\Gamma_{0}}(F)$ for $\delta_{\Gamma_{0}}(F)$ the number of nodes $p \in \Gamma_{0} \cap \Gamma_{0}^{c}$ such that the stalk of $F$ at $p$ fails to be locally free.

\begin{definition} \label{phistab}	
	Given $\phi \in V^d(\Gamma)$, we define a rank~$1$ torsion-free sheaf $F$  of degree $d$ on a nodal curve $C/k$ over an algebraically closed field to be \emph{$\phi$-(semi)stable} if
	\begin{equation} \label{Eqn: SymDefOfStability}
		\left| \deg_{\Gamma_0}(F)- \sum \limits_{v \in \operatorname{Vert}(\Gamma_0)} \phi(v) + \frac{\delta_{\Gamma_0}(F)}{2} \right| <  \frac{\#(\Gamma_0 \cap \Gamma_0^{c})-\delta_{\Gamma_0}(F)}{2}  \  \text{ (resp.~$\leq$).}
	\end{equation}
	for all proper subgraphs  $\emptyset \subsetneq \Gamma_0 \subsetneq \Gamma$.

We define $\phi \in V^d(\Gamma)$ to be \emph{nondegenerate} if every $\phi$-semistable sheaf is $\phi$-stable. We say that $\phi \in V_{g,n}^d$ is \emph{nondegenerate} if for all $\Gamma \in {G}_{g,n}$, the $\Gamma$-component $\phi(\Gamma)$ is nondegenerate in $V^d(\Gamma)$.
\end{definition}

\begin{definition} \label{Def: moduliStacks}
	Given $\phi \in V_{g,n}^d$ we say that a family of rank~$1$ torsion-free sheaves of degree $d$ on a family of nodal curves is \emph{$\phi$-(semi)stable} if Equation~\eqref{Eqn: SymDefOfStability} holds on all geometric fibers.  We define $\Jb_{g,n}^{ \text{pre}}(\phi)$ to be the category fibered in groupoids whose objects are  tuples $(C, p_1, \ldots, p_n; F)$ consisting of a family of stable $n$-pointed curves $(C/T, p_1, \ldots, p_n)$ of genus $g$, and a family of $\phi$-semistable rank~$1$ torsion-free sheaves $F$ of degree $d$ on $C/T$.  The morphisms of $\Jb_{g,n}^{\text{pre}}(\phi)$ over a $k$-morphism $t \colon T \to T'$ are pairs consisting of an isomorphism of pointed curves $\widetilde{t} \colon (C, p_1, \ldots, p_n) \cong (C'_{T}, (p'_1)_{T}, \ldots, (p'_n)_{T})$, and an isomorphism of $\mathcal{O}_{C}$-modules $F \cong \widetilde{t}^{*}(F'_{T})$.
	
	For every object $(C, p_1, \ldots, p_n; F)$ of $\Jb_{g,n}^{\text{pre}}(\phi)(T)$ the rule that sends $g \in \mathbb{G}_{m}(T)$ to the automorphism of $F$ defined by multiplication by $g$ defines an embedding $\mathbb{G}_{m}(T) \to \Aut(C, p_1, \ldots, p_n; F)$  that is compatible with pullbacks.  The image of this embedding is contained in the center of the automorphism group, so the rigidification stack is defined, and we call this stack the \emph{$\phi$-compactified universal Jacobian $\Jb_{g,n}(\phi)$}.
\end{definition}

Theorem~\ref{pic=stab} combined with Simpson's formalism implies that, for all nondegenerate $\phi$'s, the $\phi$-compactified universal Jacobian is a proper Deligne-Mumford stack.  We now show how to achieve this.

Let $(C, p_1, \ldots, p_n) \in \Mb_{g,n}$ and $A,M \in \Pic(C)$ with $A$ ample, and let $a$ (respectively $m$) be the total degree of $A$ (respectively of $M$). It follows from \cite[Inequality (10)]{kasspa} that if $\phi(A,M)$ is defined by the formula
\begin{equation} \label{phiequivtwisted}
\phi(A,M) := \frac{ (d+ 1-g+ m)}{a} \cdot \deg(A)+ \frac{1}{2} \cdot \deg(\omega_C) - \deg(M),
\end{equation}
and $F$ is any rank~$1$ torsion-free sheaf of degree $d$ on $C$, then $F$ is $\phi$-(semi)stable if and only if $F \otimes M$ is slope (semi)stable (in the sense of slope/Gieseker-stability) with respect to $A$.

In \cite[page~10]{cmkv}, the authors proved that for every $\phi \in V^d(\Gamma_C)$ there exist $A,M$ as above such that $\phi = \phi(A,M)$. Reasoning in the same way, and employing Theorem~\ref{pic=stab}, we can prove that the same holds over $\Mb_{g,n}$.

\begin{corollary} \label{corollary} Let $\phi \in V_{g,n}^d$. Then there exist line bundles $A,M$ on the universal curve $\Cb_{g,n} \to \Mb_{g,n}$ with $A$ ample relative to $\Mb_{g,n}$ such that, for every stable curve $(C, p_1, \ldots, p_n)$, a rank~$1$ torsion-free sheaf $F$  of degree $d$ on $C$ is $\phi$-semistable if and only if $F \otimes M$ is $A$-(semi)stable.
\end{corollary}

\begin{proof}
We first observe that we can reduce to the case when $\phi$ has rational coefficients.  In order to do that, we claim that for all $\phi \in V_{g,n}^d$ there exists $\phi_{\epsilon} \in V_{g,n}^0$ such that $(\phi+\phi_{\epsilon})$ has rational coefficients, and $\phi$-(semi)stability is equivalent to $(\phi+\phi_{\epsilon})$-(semi)stability. This follows immediately from the fact that the locus of degenerate $\phi$'s in $V_{g,n}^d$ consists of a locally finite union of hyperplanes, a fact that we will later prove in Remark~\ref{locallyfinite2}.% in Section~\ref{Section: Walls}.% Theorem~\ref{stabilityspace}, a result we prove later in Section~\ref{Section: Walls}, which in particular asserts that the locus of degenerate $\phi$'s in $V_{g,n}^d$ consists of a locally finite union of subspaces.

 Define $M := \omega_{\pi}(p_1+ \ldots + p_n)^{-t}$ for $t >>0$ a sufficiently large integer such that the inequality
 \begin{equation} \label{positive}
 \phi(\Gamma_C)(v) + \deg_{C_v} (M) - \frac{\deg_{C_v}(\omega_C(p_1+ \ldots + p_n))}{2} >0
 \end{equation}
 holds for all $(C, p_1, \ldots, p_n) \in \Mb_{g,n}$ and for all vertices $v$ of $\Gamma_C$. That such $t$ exists follows from the fact that $\omega_{\pi}(p_1+\ldots+p_n)$ is ample relative to $\Mb_{g,n}$, from the fact that the multidegree of a line bundle on $\Cb_{g,n}$ is the same for curves with the same dual graph, and from the fact that the set ${G}_{g,n}$ of dual graphs modulo isomorphism is finite.

For simplicity, denote by $m$  the total degree of $M$. For any integer $e$, the parameter $\psi \in V^{e}_{g,n}$ defined by
\begin{equation} \label{amplelb} \psi(\Gamma_C) := \frac{e}{2d + 2m - (2g-2+n)} \cdot \left( 2 \phi(\Gamma_C) + 2 \deg_{C} (M) -  \deg_{C}(\omega_{\pi}(p_1+ \ldots + p_n)) \right)\end{equation}
has rational coefficients (because $\phi$ does), and so by Theorem~\ref{pic=stab} it is equal to $\deg(A)$ for some rational line bundle $A \in \picrel_{g,n}^{e}(\mathbb{R})$.  By taking $e$ to be sufficiently divisible, we can clear denominators so that $A$ is an \emph{integral} line bundle.
  Moreover, by possibly replacing $e$ with $-e$, we can assume that \[\frac{e} {2d+2m -(2g-2+n)}>0\] holds. Because Inequality~\eqref{positive} holds for all  geometric points of $\Mb_{g,n}$, we deduce that $A$ is ample relative to $\Mb_{g,n}$.

For all $(C, p_1, \ldots, p_n) \in \Mb_{g,n}$ we  have then $\phi(\Gamma_C)=\phi(A,M)(\Gamma_C)$, where the latter is defined by Formula~\eqref{phiequivtwisted}, and this concludes our proof.
\end{proof}

By combining Corollary~\ref{corollary}, \cite[Proposition 3.30]{kasspa} (Simpson's representability result \cite[Theorem 1.21]{simpson} rewritten in our language) and \cite[Lemma 3.33]{kasspa}, we deduce the following.
\begin{corollary} \label{Cor: JbExists}
	Let $\phi \in V_{g,n}^d$ be nondegenerate.  Then $\Jb_{g, n}(\phi)$ is a $k$-smooth Deligne--Mumford stack, and  the forgetful morphism $\Jb_{g,n}(\phi) \to \Mb_{g,n}$ is representable, proper and flat.
\end{corollary}

 Following existing literature we will refer to $\Jb_{g,n}(\phi)$ as a \emph{fine} $\phi$-compactified universal  Jacobian. The authors expect that when $\phi$ is degenerate $\Jb_{g,n}(\phi)$ can naturally be given the structure of an Artin stack.

\begin{remark} \label{compareoldpaper2} Corollary~\ref{Cor: JbExists} generalizes our previous \cite[Corollary 3.41]{kasspa} by removing the hypothesis $n\geq 1$ and allowing for any degree $d$ (not necessarily $d=g-1$).

Moreover, by varying $\phi$ in $V_{g,n}^d$, in Corollary~\ref{Cor: JbExists} we have described all proper extensions of $\phi$-compactified universal Jacobians from moduli of treelike curves %(see Remark~\ref{compareoldpaper})
to $\Mb_{g,n}$. %In \cite{kasspa} we have already observed (\cite[Lemma 3.26, 3.32]{kasspa}) that when $n \geq 1$, for every nondegenerate $\psi \in V_{g,n}^{\text{TL}}$  there exists a nondegenerate $\phi \in V_{g,n}^{g-1}$ such that $q(\phi) = \psi$.  For each nondegenerate $\phi \in q^{-1}(\psi)$ we have now constructed an extension of $\Jb_{g,n}(\psi) \to \Mm_{g,n}^{\text{TL}}$ to $\Jb_{g,n}(\phi) \to \Mb_{g,n}$.
We observed in Remark~\ref{compareoldpaper} that such extensions are parameterized by the nondegenerate elements of $D_{g,n}$. The subspaces of  $D_{g,n}$ corresponding to degenerate elements will be explicitly described in \eqref{otherwalls}.
\end{remark}

\begin{remark} \label{esteves} Esteves constructed in \cite{esteves} the compactified Jacobian of a family of reduced curves over a scheme. Building on his work,   Melo constructed the corresponding compactified universal Jacobians over $\Mb_{g,n}$ in \cite{melo}.

In their formalism, a compactified universal Jacobian in degree $d$ is is defined in terms of a (universal) $d$-polarization, which is defined to be a vector bundle $\mathcal{E}$ of rank $r$ and degree $r(d+1-g)$ on the universal curve $\pi\colon\Cb_{g,n}\to \Mb_{g,n}$.
We claim that, for rank~$1$ torsion-free sheaves of degree $d$, the notion of $\mathcal{E}$-(semi)stability of Esteves--Melo \cite[Definition 2.9]{melo} coincides with the notion of $\phi(\mathcal{E})$-(semi)stability that we gave in Definition~\ref{phistab} after posing
\begin{equation} \label{dictionary} \phi(\mathcal{E}):= \frac{\deg(\mathcal{E})}{r} + \frac{\deg(\omega_{\pi})}{2} \quad \in V_{g,n}^d.\end{equation}
The claim follows immediately by observing that, because the total degree is fixed, the lower bound for $\deg_{C_0}(F)$ of \cite[Inequality 2.1]{melo} on all subcurves $C_0$ of $C$ is equivalent to Inequality~\eqref{Eqn: SymDefOfStability} involving $\deg_{\Gamma_{C_0}}(F)$ on all proper subgraphs $\emptyset \subsetneq \Gamma_{C_0} \subsetneq \Gamma_C$. As a consequence of our claim, when $\phi(\mathcal{E})$ is nondegenerate, $\Jb_{g,n}(\phi(\mathcal{E}))$ and Melo's moduli stack $\Jb_{g,n}^{\mathcal{E}, \text{ss}}$ are isomorphic as Deligne--Mumford stacks over $\Mb_{g,n}$.

Formula~\eqref{dictionary} shows that every $d$-polarization $\mathcal{E}$ can be translated into a $\phi$-stability condition. We deduce the converse as an easy consequence of Theorem~\ref{pic=stab}. We only discuss the case when $g \geq 2$ (the remaining cases are similar and easier).

For a given $\phi' \in V_{g,n}^d$, Theorem~\ref{pic=stab} implies that there exist $L \in \pic^0(\Cb_{g,n})$ and $\mathbb{N} \ni e >>0$ such that
\[
\phi' - \phi_{\text{can}}^d = \frac{\deg(L)}{e}.
\]
Defining the $d$-polarization by
\[
\mathcal{E} := \omega_{\pi}^{\otimes e(d+1-g)} \otimes L^{\otimes(2g-2)} \oplus \mathcal{O}^{\oplus ((2g-2) e-1)}
\]
we have that $\phi'= \phi(\mathcal{E})$, and this completes the proof of our claim.
\end{remark}

\section{The stability subspaces and polytopes}

\label{Section: Walls}

In this section we describe how  $\phi$-stability depends on $\phi$.  The stability space $V^d_{g,n}$ naturally decomposes into \emph{stability polytopes} defined by the property that two stability parameters lie in a common polytope if and only if they define the same set of stable sheaves.  The main result is Theorem~\ref{stabilityspace}, where we explicitly describe the ``\emph{stability subspaces}'', the affine linear subspaces that define the stability polytopes, in terms of the isomorphism $V_{g,n}^d \cong C_{g,n} \oplus D_{g,n}$ of Corollary~\ref{isoalphan}.  One consequence of Theorem~\ref{stabilityspace} is that every stability polytope in $V_{g,n}^d$ is a product of a hypercube in $C_{g,n}$ and of a polytope in $D_{g,n}$. The stability subspaces are hyperplanes unless $n=0$ and $\gcd(d+1-g,2g-2) = 1$, in which case some of them are empty and some coincide with $V_{g,0}^d$  (see Remark~\ref{nomarkedpoints}). %In fact Theorem~\ref{stabilityspace} describes the degenerate locus of $V_{g,n}^d$ as a union of stability subspaces which can, when $n=0$ and for certain $d$'s, coincide with $V_{g,0}^d$ (see Remark~\ref{nomarkedpoints}).

To begin, we recall some notation from  \cite[Section 3.2]{kasspa}.

We say that a subgraph $\emptyset \subsetneq \Gamma_0 \subsetneq \Gamma$ is \emph{elementary} if both $\Gamma_0$ and its complement $\Gamma_0^c$ are connected.	(The vertex set of an elementary subgraph is  an elementary cut in the sense of \cite[page~31]{oda79}.) We  now define the combinatorial objects that control stability of rank~$1$ torsion-free sheaves on a stable pointed curve whose dual graph is $\Gamma$.

	\begin{definition} \label{Definition: PolytopesOneCurve}  Let $\Gamma$ be a stable marked graph.  To a subgraph $\emptyset \subsetneq \Gamma_0 \subsetneq \Gamma$ and an integer $k \in \mathbb{Z}$ we associate the affine linear function  $\ell(\Gamma_0, k) \colon V^d(\Gamma) \to \mathbb{R}$ defined by
	\begin{equation} \label{linearfunctional}
		\ell(\Gamma_0, k)(\phi)  := k- \sum_{v \in \operatorname{Vert}(\Gamma_0)}\phi(v) + \frac{\#(\Gamma_0 \cap \Gamma_{0}^{c})}{2}.
	\end{equation}

	When $\Gamma_0$ is an elementary subgraph of $\Gamma$ we call the hyperplane
	\begin{equation} \label{stabilitypolytope}
		 H(\Gamma_0, k) := \{ \phi \in V^d(\Gamma) ;  \ell(\Gamma_0, k)(\phi) =0 \} \subsetneq V^d(\Gamma)
	\end{equation}
	a \emph{stability hyperplane}. (An element $\phi \in V^d(\Gamma)$ is nondegenerate according to Definition~\ref{phistab} if and only if $\phi$ does not belong to any such hyperplane.)

A \emph{stability polytope} in $V^d(\Gamma)$ is defined to be a connected component  of the complement of all stability hyperplanes of $V^d(\Gamma)$:
	\[
		V^d(\Gamma) - \bigcup\limits_{\substack{\emptyset \subsetneq \Gamma_0 \subsetneq \Gamma \text{ elementary}\\ k \in \mathbb{Z}}} H(\Gamma_0, k).
	\]
If $\phi_0 \in V^d(\Gamma)$ is nondegenerate, we write $\mathcal{P}(\phi_0)$ for the unique stability polytope in $V^d(\Gamma)$ that contains $\phi_0$.
\end{definition}

\begin{remark} \label{rationalbounded}
For $\phi_0 \in V^d(\Gamma)$ nondegenerate, the stability polytope  $\mathcal{P}(\phi_0)$ is a rational bounded convex polytope. Indeed, by definition we have
\begin{equation} \label{Eqn: StabilityPolytope}
	\mathcal{P}(\phi_0) = \{ \phi \in V^d(\Gamma) : \ell(\Gamma_0, k)(\phi)>0 \text{ for all } \ell(\Gamma_0, k) \text{ s.t.~}\ell(\Gamma_0, k)(\phi_0) > 0 \},
\end{equation}
and in Equation~\eqref{Eqn: StabilityPolytope}  only finitely many $\ell(\Gamma_0, k)$'s are needed to define $\mathcal{P}(\phi_0)$.
\end{remark}

\begin{remark} \label{locallyfinite} For $\Gamma$ a stable marked graph, the hyperplane arrangement consisting of the collection $\{H(\Gamma_0, k)\}_{\emptyset \subsetneq \Gamma_0 \subsetneq \Gamma, k \in \mathbb{Z}}$ is \emph{locally finite}. This follows from the fact that there are only finitely many subgraphs $\Gamma_0$ of $\Gamma$, and from Equation~\eqref{linearfunctional}.
\end{remark}

Assume $\phi_1, \phi_2 \in V(\Gamma)$ are nondegenerate. As a consequence of \cite[Lemma 3.20]{kasspa} we have that $\phi_1$-(semi)stability coincides with $\phi_2$-(semi)stability if and only if $\mathcal{P}(\phi_1) = \mathcal{P}(\phi_2)$. %(In other words, the fibers $\overline{J}_C(\phi_1)$ and $\overline{J}_C(\phi_2)$ of the moduli stacks $\Jb_{g,n}(\phi_1)$ and $\Jb_{g,n}(\phi_2)$ over a geometric point $(C,p_1, \ldots, p_n) \in \Mb_{g,n}$ whose dual graph is $\Gamma$ parameterize different sets of sheaves.)
In \cite[Example 3.21]{kasspa} we showed that this would no longer be true if  in Definition~\ref{Definition: PolytopesOneCurve} the \emph{elementary} condition for the subgraphs was dropped.

The simplest nontrivial examples of this stability decomposition occur when  $C$ has $2$ irreducible components.
\begin{example} \label{exampledollarsign} ($\phi$-stability on vine curves). Suppose that $C$ is a nodal curve consisting of $2$ smooth irreducible components of genera $i$ and $j$ connected by $\alpha$ nodes. Its dual graph $\Gamma(\alpha,i,j)$ has $2$ vertices of genera $i$ and $j$ joined by $\alpha$ edges.

The degree $d$ stability space $V^d(\Gamma)$ is the line in the plane $V(\Gamma)$ consisting of points with coordinates summing to $d$. It immediately follows from Definition~\ref{Definition: PolytopesOneCurve} that the stability hyperplane $H(\Gamma, \Gamma_0, k)$ of $V^d(\Gamma)$ is the point $(k+\alpha/2, d -k-\alpha/2)$ for $\Gamma_0$ the complete subgraph of $\Gamma$ that contains the first vertex. % (independent of the choice of the $1$-vertex subgraph $\Gamma_0$ of $\Gamma$).
%\[\begin{cases} (k, d-k) \textrm{ for all } k \in \mathbb{Z}&  \textrm{ when } \alpha \textrm{ is even}, \\
%(k+\frac{1}{2}, d-k-\frac{1}{2}) \textrm{ for all } k \in \mathbb{Z}& \textrm{ when } \alpha \textrm{ is odd.}\end{cases}
%\]

The stability polytopes are therefore the segments whose endpoints have integer (resp.~half-integer) coordinates when $\alpha$ is even (resp.~when $\alpha$ is odd). If $\phi$ belongs to the relative interior of one such segment, there are $\alpha$ distinct $\phi$-stable bidegrees of line bundles on $C$ and these bidegrees are those that are closest to $\phi$. If  $\phi$ varies from the relative interior of a segment to one of its endpoints, the stable bidegree furthest away from that endpoint becomes strictly semistable as does the  nearest unstable bidegree.
\end{example}

We now define the analogous objects for $V_{g,n}^d$. Similar to what we did in Definition~\ref{Definition: PolytopesOneCurve} for a single stable graph $\Gamma$, we introduce stability subspaces $V_{g,n}^d$ such that $\phi$ is nondegenerate according to Definition~\ref{phistab} if and only if $\phi$ does not belong to one such subspace.
\begin{definition} \label{Definition: PolytopesFamilyCurves}
	For $\Gamma$ a stable  marked graph, $\emptyset \subsetneq \Gamma_0 \subsetneq \Gamma$ an elementary subgraph, and $k \in \mathbb{Z}$ an integer, we call the subspace
	\[
		H(\Gamma, \Gamma_0, k) := \{ \phi \in V_{g,n}^d:  \ell(\Gamma_0, k)(\phi(\Gamma)) =0 \}
	\]
	 of $V_{g,n}^d$ a \emph{stability subspace}. %(An element $\phi \in V^d_{g,n}$ is nondegenerate according to Definition~\ref{phistab} if and only if $\phi$ does not belong to any such hyperplane.)

A \emph{stability polytope} in $V_{g,n}^d$ is defined to be a connected component of the complement of all stability subspaces of $V_{g,n}^d$, i.e.~a connected component of
	\[
		V_{g,n}^d - \bigcup\limits_{\substack{ \Gamma \text{ stable graph} \\\emptyset \subsetneq \Gamma_0 \subsetneq \Gamma \text{ elementary}\\ k \in \mathbb{Z}}} H(\Gamma, \Gamma_0, k)
	\]
If $\phi_0 \in V^d_{g,n}$ is nondegenerate, we write $\mathcal{P}(\phi_0)$ for the unique stability polytope in $V^d_{g,n}$ that contains $\phi_0$.

We let $\mathcal{P}^d_{g,n}$ be the set of stability polytopes in $V_{g,n}^d$ and define \emph{the stability polytope decomposition} \[\mathcal{P}_{g,n}:= \bigcup_{d \in \mathbb{Z}} \mathcal{P}^d_{g,n}.\]
\end{definition}

The subspaces $H(\Gamma, \Gamma_0, k)$ that we introduced in Definition~\ref{Definition: PolytopesFamilyCurves} are often hyperplanes, although we can have $H(\Gamma, \Gamma_0, k) = \emptyset$ or $H(\Gamma, \Gamma_0,k) = V^d_{g,n}$ (in this last case the stability polytope decomposition is empty). These degenerate subspaces only occur for $n=0$  and certain $d$'s as described in Remark~\ref{nomarkedpoints},  an application of Theorem~\ref{stabilityspace}.

\begin{remark} \label{rationalbounded2} As was the case for $V^d(\Gamma)$, the stability polytopes in $V_{g,n}^d$ are rational bounded convex polytopes.
\end{remark}

\begin{remark} \label{locallyfinite2} We can deduce that
\begin{equation} \label{hyperarr} \{H(\Gamma, \Gamma_0, k)\} \quad \textrm{for all } {\Gamma\in {G}_{g,n}, \ \emptyset \subsetneq \Gamma_0 \subsetneq \Gamma \textrm{ elementary and } k \in \mathbb{Z}} \end{equation}
is locally finite from the analogous result for the stability hyperplane arrangement of $V^d(\Gamma)$ (Remark~\ref{locallyfinite}). Indeed, this follows immediately from the fact that \eqref{hyperarr} is the restriction to $V_{g,n}^d$ of the decomposition of $\prod_{\Gamma \in {G}_{g,n}} V^d(\Gamma)$ given by products of stability hyperplanes.
\end{remark}

It follows from Lemma~\ref{Lemma: PhiDeterminedByTwoComponents} that a stability parameter $\phi \in V_{g,n}^d$ is uniquely determined by its restriction to all loopless graphs with $2$ vertices. We now prove analogous statements about stability subspaces and polytopes.

\begin{lemma} \label{Lemma: HyperplanesByTwoComponents}
	If $H(\Gamma, \Gamma_0, k) \subseteq V_{g,n}^d$ is a stability subspace, then there exists a loopless $2$-vertex stable graph $\Gamma'$ and an elementary subgraph $\Gamma_0'$ of $\Gamma'$ such that
	$
		H(\Gamma, \Gamma_0, k) = H(\Gamma', \Gamma_0', k).
	$
\end{lemma}
\begin{proof}
Starting with $\Gamma$, contract edges so that $\Gamma_0$ is contracted to a vertex $w$ and its complement $\Gamma_0^{c}$ to a vertex $w^c$, and then contract all resulting loops. Call $\Gamma'$ the resulting graph with $2$ vertices $w$ and $w^c$.  By inductively applying compatibility with contractions, we find $\phi(\Gamma')(w) = \sum_{v \in \Gamma_0} \phi(\Gamma)(v)$. This implies that 	$	 H(\Gamma, \Gamma_0, k)$ equals $H(\Gamma', \Gamma_0', k)$.
\end{proof}
%A restatement of the lemma is that $\phi$-stability can be detected by vine curves:

%\begin{corollary} \label{Cor: geomstability} Let $F$ be a family of rank $1$ torsion-free sheaves of degree $d$ on the universal curve $\Cb_{g,n} \to \Mb_{g,n}$,  and let $\phi \in V_{g,n}^d$. If the restriction of $F$ to the locus of curves with at most $2$ smooth irreducible components is $\phi$-(semi)stable, then $F$ is $\phi$-(semi)stable.
%\end{corollary}

%As another corollary of Lemma~\ref{Lemma: HyperplanesByTwoComponents}, we obtain the following partial description of the stability polytope decomposition of $V_{g,n}^d$.

%\begin{corollary} \label{Corollary: StabilityDeterminedByTwoComponents}
%	Given a stability polytope $\mathcal{P}(\Gamma) \subset V^d(\Gamma)$ for every loopless $2$-vertex stable graph $\Gamma$, there  is \emph{at most one} stability polytope $\mathcal{P} \subset V_{g,n}^d$ such that the $\Gamma$-component of $\mathcal{P}$ equals $\mathcal{P}(\Gamma)$ for all $\Gamma$.
%\end{corollary}

Using the results of Section~\ref{stabilityspacesec}, we will now explicitly write down the stability subspaces $H(\Gamma, \Gamma_0, k)$ of $V_{g,n}^d$ defined in Definition~\ref{Definition: PolytopesFamilyCurves}. By Corollary~\ref{isoalphan}, an element $\phi \in V_{g,n}^d$ is uniquely determined by its image under $(p_C \oplus p_D) \circ \rho^d$, i.e.~the projection to $C_{g,n} \oplus D_{g,n}$ of the difference $\phi - \phi_{\text{can}}^d$. We will describe the stability subspaces of $V_{g,n}^d$ as inverse images  of certain affine linear subspaces $W_C(i,S,k)$ of $C_{g,n}$ and $W_D(\ell,S,k)$ of $D_{g,n}$. %\eqref{otherwalls} of $C_{g,n}$ and of $D_{g,n}$.

We first define the subspaces of $C_{g,n}$. Denote by $(\alpha_{i,S}, -\alpha_{i,S})$ the component $\psi(\Gamma(i,S))$ of each element $\psi \in C_{g,n}$. For each triple $(i,S, k)$ with $(i,S)$ as in Section~\ref{notationmoduli} and $k \in \mathbb{Z}$, we define the hyperplane $W_C(i,S,k)$ of $C_{g,n}$ by the equations \begin{equation} \label{wallsct} W_C(i,S,k):=\begin{cases} \alpha_{i,S} =  k- \frac{(2i-1)(d+1-g)}{2g-2}  & \textrm{ when } g \geq 2 \\ \alpha_{i,S} = k-\frac{1}{2}  & \textrm{ when }  g \leq 1. \end{cases}\end{equation} %(This is a translation, depending on $d$, of an integer translate of a coordinate hyperplane).

For $g \geq 1$, we now define the relevant subspaces of $D_{g,n}$. For each element $\psi \in D_{g,n}$ we denote by $(x_j, -x_j)$ the component $\psi(\Gamma_j)$ for $j = 1+ \delta_{1,g}, \ldots, n$.
For each triple $(\ell, S, k)$ with  $0 \leq \ell < 2g-2+ \delta_{1,g},  S\subseteq [n] \setminus \{\delta_{1,g}\}$ and $k \in \mathbb{Z}$ (excluding the ``unstable''  case  $\ell =0$ and $S=\emptyset$), we define the subspace $W_D(\ell, S, k)$ of $D_{g,n}$ by the equation
\begin{equation} \label{otherwalls}
W_D(\ell, S, k) := \left\{ \vec{x} : \ x_{S} + \frac{\ell  (d+1-g-x_{[n]\setminus \{\delta_{1,g}\}})}{2g-2+ \delta_{1,g}} =k\right\}.
\end{equation}
(This subspace is a hyperplane when $n \geq 1$. When $n=0$, it can either be empty or coincide with $D_{g,0} = \{ 0 \} $. See Remark~\ref{nomarkedpoints} for more details).% could be degenerate when $n=0$).

We now show that, under the isomorphism $(p_C \oplus p_D)\circ \rho^d \colon V^{d}_{g,n} \to C_{g, n} \oplus D_{g,n}$ in Corollary~\ref{isoalphan}, the stability subspaces of $V_{g,n}^d$ defined in Definition~\ref{Definition: PolytopesFamilyCurves} correspond to the explicit subspaces just introduced.

\begin{theorem} \label{stabilityspace} The stability subspaces of $V_{g,n}^d$  are the pullbacks via $p_C \circ \rho^d$ of the hyperplanes $W_C(i,S,k)$ of $C_{g,n}$ defined in \eqref{wallsct}, and the pullbacks via $p_D \circ \rho^d$ of the subspaces $W_D(\ell, S, k)$ of $D_{g,n}$ defined in \eqref{otherwalls}.
\end{theorem}

\begin{proof}
By Definition~\ref{Definition: PolytopesFamilyCurves}, the stability subspaces $H(\Gamma, \Gamma_0, t)$ of
\[
V_{g,n}^d \subseteq \prod_{\Gamma \in {G}_{g,n}} V^d(\Gamma)
\]
are the intersections of $V_{g,n}^d$ and the obvious hyperplanes  of $\prod_{\Gamma \in {G}_{g,n}} V^d(\Gamma)$ that, by abuse of notation, we also denote by $H(\Gamma, \Gamma_0, t)$. Here $\Gamma$ varies over all stable graphs, $\Gamma_0$ varies over all elementary subgraphs of $\Gamma$, and $t \in \mathbb{Z}$. By Lemma~\ref{Lemma: HyperplanesByTwoComponents}, it is enough to consider the hyperplanes $H(\Gamma, \Gamma_{0}, t)$ for $\Gamma$ a loopless graph with $2$ vertices (rather than an arbitrary stable graph).  We prove the result by verifying that, under the isomorphism $(p_C \oplus p_D) \circ \rho^d \colon V^{d}_{g,n} \to C_{g,n} \oplus D_{g, n}$, the hyperplane $H(\Gamma, \Gamma_0, t)$ corresponds to either a hyperplane $W_C(i,S,k) \oplus D_{g,n}$ or a subspace (often also a hyperplane) $C_{g,n} \oplus W_D(\ell,S,k)$.  %When $\Gamma$ is of the form $\Gamma(i, S)$ (corresponding to the $C_{g, n}$ factor) or $\Gamma_j$ (corresponding to the $D_{g,n}$ factor), this is straightforward, but some additional work is needed when $\Gamma$ is another loopless graph with $2$ vertices.

To do this, we compute $\phi(\Gamma)$ in terms of \[(p_C \oplus p_D) \circ \rho^d (\phi)= (p_C \oplus p_D) (\phi- \phi_{\text{can}}^d) \in C_{g,n} \oplus D_{g,n}\] for all loopless graphs $\Gamma$ with $2$ vertices. To this end, we let $\psi : = \phi -\phi^d_{\text{can}} \in V^0_{g,n}$ and write $\psi(\Gamma(i,S)) = (\alpha_{i,S}, - \alpha_{i,S})$ and   $\psi(\Gamma_j) = (x_{j}, -x_{j})$ for all $j=1, \ldots, n$. We assume $g \geq 2$ to simplify the notation (the cases $g=0,1$ are similar).

%Recall the isomorphism $V^{d}_{g, n} \cong C_{g, n} \oplus D_{g, n}$ is defined by sending $\phi$ to the projection of $\psi := \phi - \phi_{\text{can}}^{d}$ onto the relevant components.

If $\Gamma=\Gamma(i, S)$ is a stable graph with $2$ vertices and $1$ edge, then every associated stability subspace is of the form $H(\Gamma, \Gamma_0, t)$, for some $t\in \mathbb{Z}$ and $\Gamma_0$ the first vertex of $\Gamma$ (according to the convention in Section~\ref{notationmoduli}).  The component $\phi(\Gamma(i, S))$ is given by
\begin{equation} \label{blah}
	\phi(\Gamma(i,S)) = %\begin{cases}
		\left(\alpha_{i,S} + \frac{d}{2g-2} (2i-1), - \alpha_{i,S} + \frac{d}{2g-2} (2g-2i-1)\right),% &	 \textrm{ when } g \geq 2 \\
		%\left(\alpha_{i,S} +d , -\alpha_{i,S} \right)& 									\textrm{ when } g \leq 1,
	%\end{cases}
\end{equation}
and we conclude that $H(\Gamma, \Gamma_0, t)$ corresponds to the hyperplane  $W_C(i, S, t+1-i) \oplus D_{g,n}$.%  (or to $W(i, S, k=t+1-d)$ when $g \leq 1$).

%The case where $\Gamma=\Gamma_j$ for some $j=1+ \delta_{1,g},\ldots, n$ is similar.
 Suppose now that $\Gamma$  is any other loopless graph with $2$ vertices and say that its first vertex has genus~$i$, contains the markings $S$, and is joined to the second vertex by $\alpha \geq 2$ edges.  Assume that $\Gamma_0$ is the first vertex of $\Gamma$. %We assume $\alpha >1$ since otherwise $\Gamma=\Gamma(i,S)$. Furthermore, we can assume that $\Gamma_0 = \{ v_{1} \}$ in the  equation $0=\ell(\Gamma_{0}, k)(\phi(\Gamma))$.

The component $\psi(\Gamma)$ can be expressed in terms of the $\alpha_{i, S}$'s and the $x_j$'s  as
\begin{equation} \label{explicitformula}
	 \phi(\Gamma) = \left(x_{S}+ \frac{2i-2 +\alpha}{2g-2} \cdot (d- x_{[n]}), \  x_{S^c}+ \frac{2g-2i - \alpha}{2g-2}\cdot (d-x_{[n]}) \right).
\end{equation}
Indeed,
apply Formulas~\eqref{one}, \eqref{two}, \eqref{three} and \eqref{four} and invert the corresponding matrix to find that $\psi = \deg(L)$ for $L \in \picrel_{g,n}^0(\mathbb{R})$ defined by
\begin{equation} \label{changebasis}
L:=  \sum_{(i,S)}  \left(\alpha_{i,S} + \sum_{j \in S} \frac{2i+1-2g}{2g-2} x_j + \sum_{j \in S^c} \frac{1-2i}{2g-2} x_j \right) \cdot C_{i,S}^- + \sum_{j=1}^n \frac{x_j}{2g-2} \cdot T_j,
\end{equation}
 then compute $\psi(\Gamma)$ as the bidegree of $L$ on a curve whose dual graph is $\Gamma$, and finally determine $\phi(\Gamma)$ as $\psi(\Gamma) + \phi^d_{\text{can}}(\Gamma)$. % translate back by $\phi_{\text{can}}^d$.

Using Formula~\eqref{explicitformula}, we deduce that the stability subspace  $H(\Gamma, \Gamma_0, t)$ corresponds to the subspace $C_{g,n} \oplus W_D(2i-2+\alpha, S, t+1-i)$. % from~\eqref{otherwalls}.
\end{proof}

\begin{remark} When $\ell = 0$, the subspaces $W_D(\ell, S, k)$ of $D_{g,n} \cong \mathbb{R}^{n-\delta_{1,g}}\ni (x_{1 + \delta_{1,g}}, \ldots, x_n)$ are independent of $d$, and take the form
\[
W=W_D(\ell, S,k)= \left\{\sum_{j \in S} x_j = k \right\}.
\]
When $n \geq 1$ the collection $\{ W_D(\ell, S, 0) \}$ is known in the literature as the \emph{resonance hyperplane arrangement}, see \cite{cjm} and \cite{ssv}. For $n \geq 1$ and for all $d \in \mathbb{Z}$, the hyperplane arrangement of $D_{g,n}$ described by Equation~\eqref{otherwalls} is therefore a refinement of the integer translates of the resonance hyperplane arrangement. (These two arrangements coincide when $g=1$).\label{pubdescription0}
\end{remark}

\begin{remark}
For $g \geq 1$ and for any $d \in \mathbb{Z}$, each translation $x_i \mapsto x_i + (2g-2+ \delta_{1,g})$ respects the collection of stability hyperplanes~\eqref{otherwalls} of $D_{g,n}$. The same is true of each translation by $(0, \ldots, \pm 1, \ldots, \mp 1 , \ldots)$. A fundamental domain for these translations is any $(2g-2+\delta_{1,g})\times 1 \ldots \times 1$ top-dimensional hyperrectangle in $D_{g,n}$, so the hyperplane arrangement~\eqref{otherwalls} naturally defines a hyperplane arrangement of the torus obtained by identifying the opposite faces of the hyperrectangle. This fact will play a role in the next section in Lemma~\ref{fundamentaldomain}. \label{pubdescription1}
\end{remark}

\begin{remark} In studying the stability polytope decomposition of $V_{g,n}^d$ for all $d \in \mathbb{Z}$, it is enough to analyze the cases $d=0,\ldots, g-1$. Indeed, tensoring with $\omega_{\pi}$ and possibly mapping $L \mapsto L^{-1}$ gives isomorphisms $\picrel_{g,n}^d(\mathbb{Z}) \to \picrel_{g,n}^{2g-2\pm d}(\mathbb{Z})$.

Furthermore, we observe that for any $(d_1, \ldots, d_n) \in \mathbb{Z}^n$, the affine endomorphism $V_{g,n}^{e} \to V_{g,n}^{e+ \sum d_j}$ defined by $\phi \mapsto \phi + \deg(d_1 \Sigma_1 + \ldots + d_n \Sigma_n)$ respects the subspaces~\eqref{wallsct} and \eqref{otherwalls}. This fact will play an important role in Lemma~\ref{key}. \label{pubdescription2}
\end{remark}

In light of Remarks~\ref{pubdescription1} and \ref{pubdescription2}, in Figures~\ref{Figure1} and \ref{Figure2} we give a picture of the spaces $D_{g,n}$ and of their stability decompositions when $\dim(D_{g,n}) = 2$ and $g \leq 3$.

\begin{figure}[ht]
    \centering
    \includegraphics[scale=0.3]{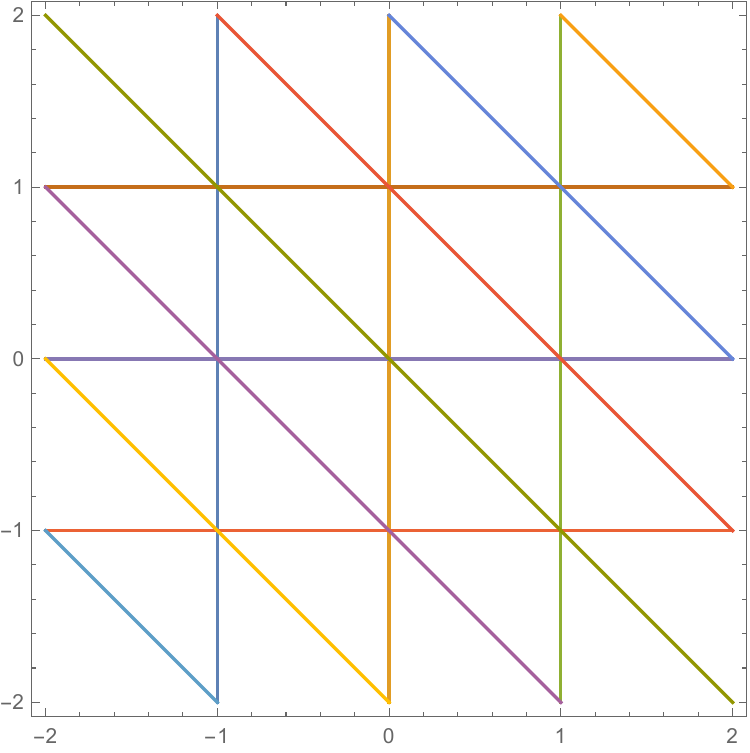} \
      \includegraphics[scale=0.3]{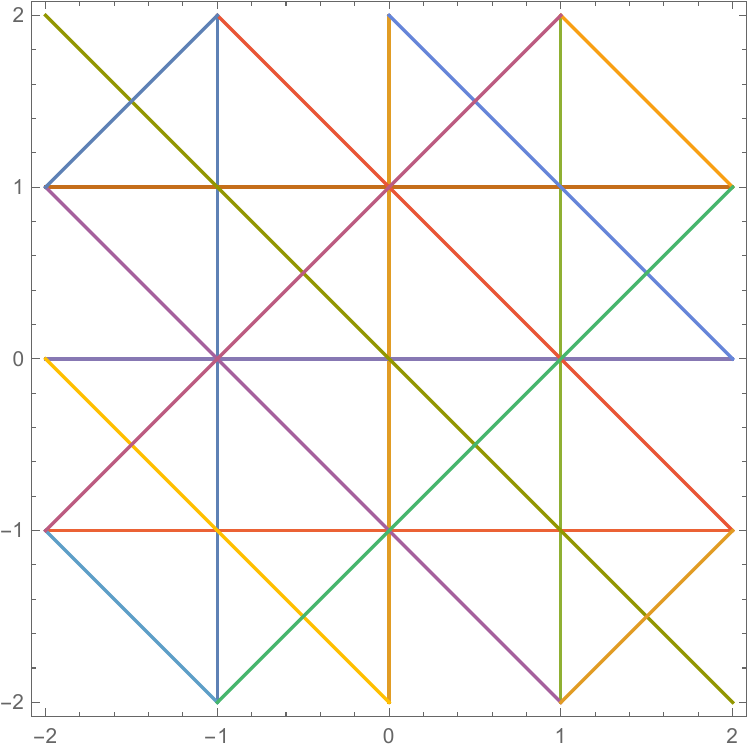} \
        \includegraphics[scale=0.3]{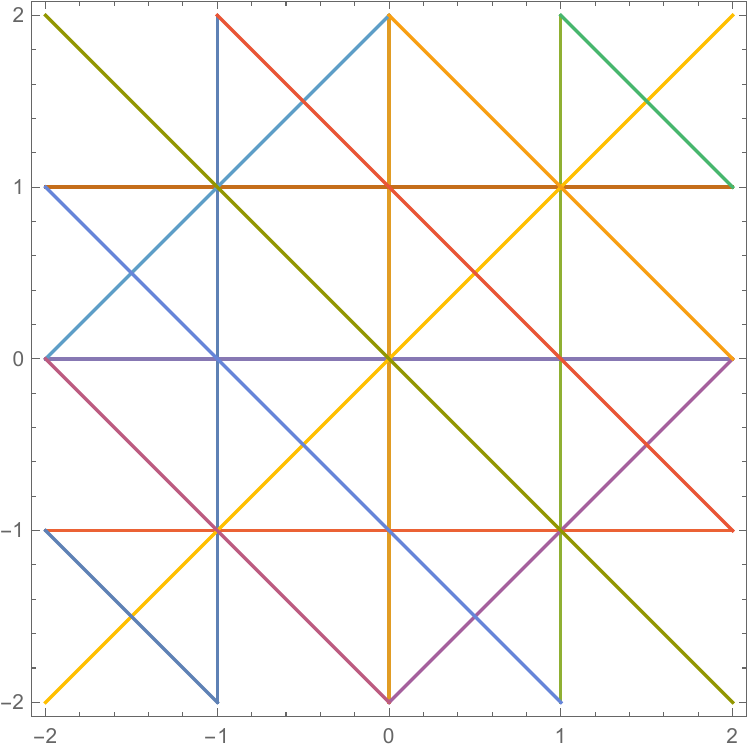}
    \caption{The stability space $D_{1,3}$ (any $d$) and the spaces $D_{2,2}$ in degrees $d=0,1$ respectively.}
    \label{Figure1}
\end{figure}

\begin{figure}[ht]
    \centering
    \includegraphics[scale=0.3]{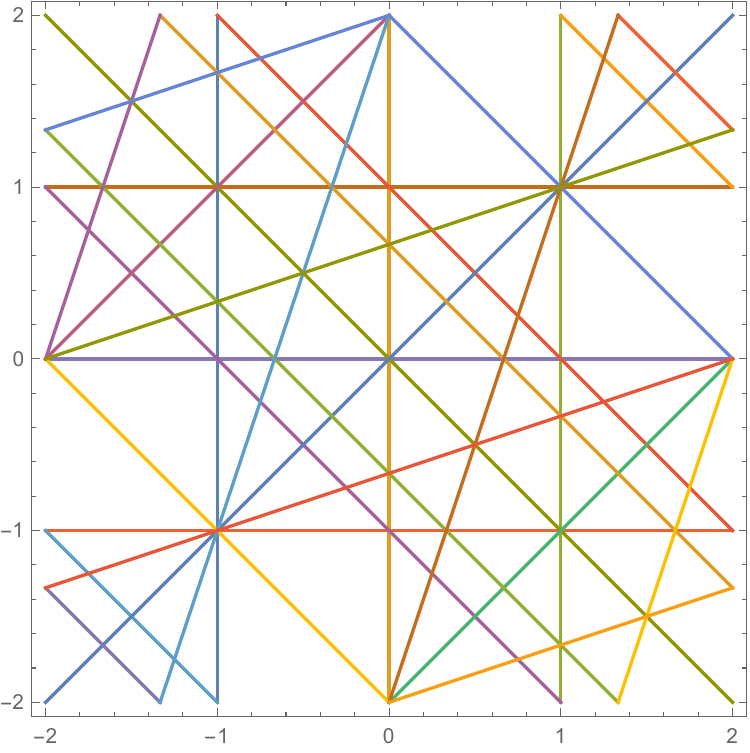} \
      \includegraphics[scale=0.3]{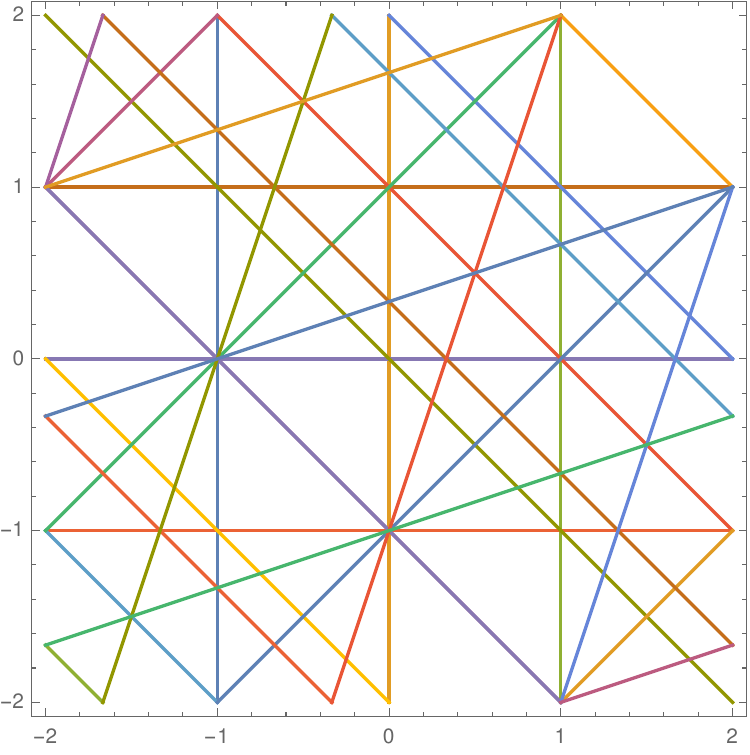} \
        \includegraphics[scale=0.3]{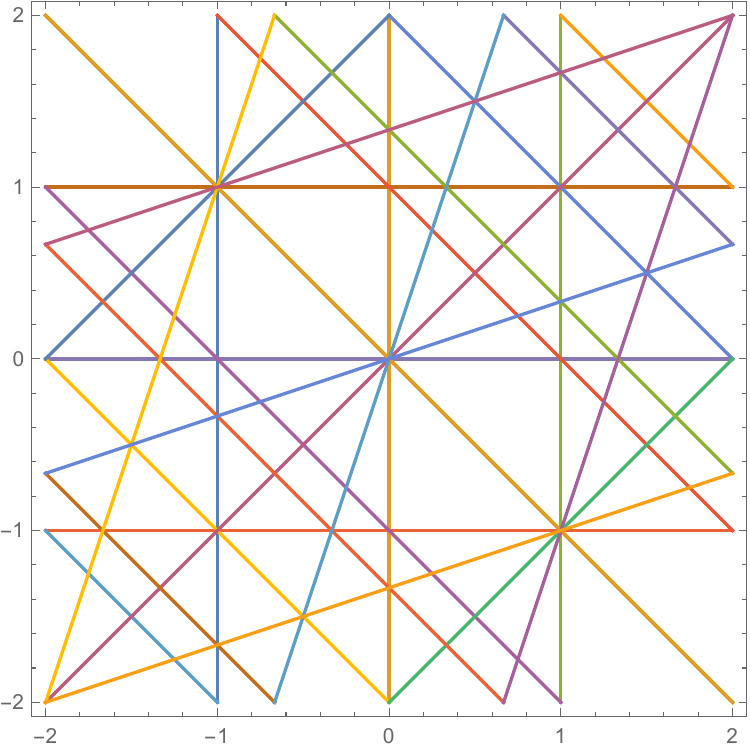}
    \caption{The stability spaces $D_{3,2}$ in degrees $d=0,1,2$ respectively.}
    \label{Figure2}
\end{figure}

Some important concluding remarks are in order.

\begin{remark} \label{nomarkedpoints} Here we analyze when the stability subspaces of $V_{g,n}^d$ are not hyperplanes. Applying Theorem~\ref{stabilityspace}, all stability subspaces $H(\Gamma, \Gamma_0, t)$ coincide with some $W_C(i, S, k)$ as in \eqref{wallsct} or with some $W_D(\ell, S, k)$ as in \eqref{otherwalls}. As we have already observed, the former are always hyperplanes. Moreover, from the expression in \eqref{otherwalls} it follows immediately that when $n \geq 1$ the subspaces $W_D(\ell, S, k)$ are also hyperplanes.

When $n=0$ we have $D_{g,0}^d \cong \{0 \}$, so $W_D(\ell, S, k)$ is a hyperplane if and only if it is empty, and from Equation~\eqref{otherwalls} this occurs if and only if
\begin{equation} \label{degeneratewall}
\frac{\ell (d+1-g)}{2g-2}  \notin \mathbb{Z}.
\end{equation}

One can further ask when $V_{g,n}^d$ contains nondegenerate elements. By the above analysis, this is always the case when $n \geq 1$. When $n=0$, this happens if and only if $W_D(\ell, S, k) = \emptyset$, which happens if and only if
\begin{equation} \label{OSnotfine}
\quad \frac{\ell (d+1-g)}{2g-2}  \notin \mathbb{Z} \text{ for all } 0 < \ell \leq g-1.
\end{equation}
%In other words, a  \emph{fine} $\phi$-compactified universal Jacobian of degree $d$ on $\Mb_{g}$ exists if and only if Condition~\eqref{OSnotfine} holds.
It is elementary to check that Condition~\eqref{OSnotfine} is equivalent to the condition that $d+1-g$ and $2g-2$ have no nontrivial common divisors. In other words, the condition
\begin{equation} \label{caporasonotfine}
\gcd(d+1-g, 2g-2) =1.
\end{equation}
\end{remark}

\begin{remark} \label{phicandegenerate} Here we analyze when $\phi_{\text{can}}^d$ belongs to a stability subspace.

We begin by observing that this does not depend on the number $n$ of marked points. As we proved in Remark~\ref{nomarkedpoints}, the projection of $\phi_{\text{can}}^d$ to $D_{g,n}$ does not belong to a stability subspace if and only if Condition~\eqref{OSnotfine} is satisfied. The projection of $\phi_{\text{can}}^d$ to $C_{g,n}$  belongs to a stability subspace if and only if $d(2i-1)$ is an odd multiple of $g-1$ for some $0 \leq i \leq g$.

  We conclude that the canonical parameter $\phi_{\text{can}}^d$ is in the interior of a stability polytope in $V_{g,n}^d$ if and only if Equation~\eqref{caporasonotfine} is satisfied. %The other extreme cases occur when $\phi_{\text{can}}^d$ is a vertex of the polytope decomposition of $V_{g,n}^d$. By the above description, this occurs if and only if $d=g-1 + \ell (2g-2)$, the degrees we have studied in \cite{kasspa}.
\end{remark}

\begin{remark}	\label{Remark: RelationWithCaporasoPand}
	 The stacks $\Jb_{g,n}(\phi)$ are related to the compactified universal Jacobians constructed by  Caporaso \cite{caporaso, caporaso08a}, Melo \cite{melo09}, and Pandharipande \cite{panda}.  In \cite{melo09}, extending earlier work of Caporaso \cite[Theorem~5.9]{caporaso08a}, Melo proved that a certain rigidified quotient stack $\overline{\mathcal{P}}^{d}_{g}$, constructed in \cite{caporaso08a}, is an Artin stack that extends $\J_{g,n}^{d}$ \cite[Theorem~3.1, Proposition~4.2]{melo09}.  Work of Esteves--Pacini \cite{espac} shows that $\overline{\mathcal{P}}^{d}_{g}$ is isomorphic to $\Jb_{g,0}(\phi_{\text{can}}^{d})$.  This is not stated in the paper, but the context is essentially \cite[Theorem~6.3]{espac}.%, which is the analogous result for the functor of isomorphism classes of objects in $\Jb_{g,n}(\phi_{\text{can}})$.
	
	 The isomorphism $\overline{\mathcal{P}}^{d}_{g} \cong \Jb_{g,0}(\phi_{\text{can}}^{d})$ is constructed as follows.  The stack $\overline{\mathcal{P}}^{d}_{g}$ is the rigidification of the stack parameterizing balanced line bundles on quasistable curves. (See \cite[Definition~1, Section~1.1.3]{melo09} for exact definitions). An essential feature of these definitions is that they the imply \cite[Proposition~5.4]{espac}.  That proposition states that the direct image of a family of balanced line bundles under the stabilization map $C \to C^{\text{st}}$ is a family of rank $1$, torsion-free sheaves, and the formation of this direct image commutes with base change.  A computation of stability conditions shows that the direct image is a family of $\phi^{d}_{\text{can}}$-semistable sheaves \cite[Proposition~6.2]{espac}.  We conclude that taking the direct image defines a morphism $\overline{\mathcal{P}}^{d}_{g} \cong \Jb_{g,0}(\phi^{d}_{\text{can}})$, and a pseudoinverse is defined using \cite[Proposition~5.5]{espac}.

We can also apply the argument from the previous paragraph when $n>0$, and then it produces a morphism $\Jb_{g,n}(\phi^{d}_{\text{can}}) \to \overline{\mathcal{P}}^{d}_{g}$.  Indeed, the essential point is that every $\phi^d_{\text{can}}$-semistable sheaf is admissible (in the sense of  \cite[Section~3]{espac}) with respect to $C \to C^{\text{st}}$  (i.e.~the total degree of a $\phi^d_{\text{can}}$-semistable sheaf on a rational chain is always $-1$, $0$, or $+1$).  We conclude from loc.~cit.~that  the direct image of a family of $\phi^d_{\text{can}}$-semistable rank~$1$ torsion-free sheaves is a family of rank~$1$ torsion-free sheaves whose formation commutes with base change. For line bundles this is \cite[Theorem~3.1]{espac}, and the general case can be deduced from the line bundle case using \cite[Proposition~5.2]{espac}.  Finally, a computation of stability conditions shows that the direct image is $\phi^d_{\text{can}}$-semistable.

For some choices of $d$, the stability parameter $\phi^{d}_{\text{can}}$ will be degenerate, but we can then relate $\overline{\mathcal{P}}^{d}_{g}$ to a nondegenerate stability parameter.  If $\phi_{\epsilon}$ is nondegenerate and sufficiently close to $\phi^{d}_{\text{can}}$, then every $\phi_{\epsilon}$-stable sheaf is $\phi^{d}_{\text{can}}$-semistable, so  there is a tautological morphism $\Jb_{g, n}(\phi_{\epsilon}) \to \Jb_{g, n}(\phi^d_{\text{can}})$.  The stack $\Jb_{g, n}(\phi_{\epsilon})$ is Deligne--Mumford, and by composition, we get a morphism
\begin{equation} \label{Eqn: JbarToPand}
	\Jb_{g, n}(\phi_{\epsilon}) \to  \overline{\mathcal{P}}^{d}_{g}
\end{equation}
that lifts the forgetful morphism $\Mb_{g, n} \to \Mb_{g,0}$.

\end{remark}

\section{Applications}
\label{final}
In this section we apply our earlier results in two ways. We study the problem of extending to $\Mb_{g,n}$ the sections of the forgetful map $\J_{g,n}^d \to \Mm_{g,n}$, and the problem of finding different isomorphism classes of fine $\phi$-compactified universal Jacobians.

For fixed integers $(k; d_1, \ldots, d_n)$ satisfying $k (2-2g) + d_1 + \ldots + d_n  = d$, we define a natural map $\sigma_{k, \vec{d}} \colon \Mm_{g, n} \to \J_{g,n}^{d}$ by the rule
 \begin{equation} \label{Eqn: DefOfAbelSection}
	\sigma_{k, \vec{d}} \colon (C/S, \Sigma_1, \ldots, \Sigma_n) \mapsto   \omega_{\pi} ^{\otimes -k} \otimes \mathcal{O}_C( d_1 \Sigma_1 + \ldots +d_n \Sigma_n).
\end{equation}
This section is sometimes called an \emph{Abel--Jacobi section}. These are the only rational sections of the forgetful map $\J_{g,n}^d \to \Mm_{g,n}$ from the universal Jacobian of degree $d$ to the moduli space of smooth pointed curves, by the following result, which motivates Section~\ref{Section: AbelJacobi}.

\begin{fact} \label{franchetta} (Strong Franchetta conjecture.) Every rational section of the forgetful map $\J_{g,n}^d \to \Mm_{g,n}$ is of the form $\sigma_{k, \vec{d}}$ for some $k$ and $\vec{d}$ as above.  In particular, every rational section extends to a regular section $\Mm_{g,n} \to \J_{g,n}^d$.
\end{fact}
\begin{proof}
When $n \geq 1$ this is a well-known consequence of Fact~\ref{arbarellocornalba}, see \cite[Section 4]{acpicard}. When $n=0$ this was proven  by Mestrano \cite{mestrano} and then by Kouvidakis \cite[Theorem 2]{kouvi}.
\end{proof}

Fact~\ref{franchetta} implies the following result on birational maps of universal Jacobians, which motivates Section~\ref{Section: different}.
\begin{corollary} \label{corfranchetta} Let $\alpha \colon \J_{g,n}^{e_1} \dashrightarrow \J_{g,n}^{e_2}$ be a birational map that commutes with the forgetful maps to $\Mm_{g,n}$. Then there exist \[(k; d_1, \ldots, d_n) \in \mathbb{Z}^{n+1} \textrm{ and } t \in \{0, 1\} \quad \text{with} \quad  k(2-2g)+ \sum d_j  = e_2- (-1)^t e_1,\] such that $\alpha$ is defined by the rule
\begin{equation} \label{bir}
 \alpha \colon L \mapsto L^{(-1)^t} \otimes  \omega_{C}^{\otimes -k} \otimes \mathcal{O}_C( d_1 \Sigma_1 + \ldots + d_n \Sigma_n).
\end{equation}

In particular, $\alpha$ is an isomorphism.
\end{corollary}
\begin{proof}
The case $n=0$ is due to Caporaso, see \cite[Theorem 7.2]{bfv}. From now on in this proof we assume $n \geq 1$.

By applying a translation automorphism, it is enough to prove the claim when $e_1=e_2=0$, so $\alpha$ is a birational automorphism of $\J_{g,n}^0$ that commutes with the forgetful map.  In this case, consider the birational automorphism $\beta$ of the generic Jacobian $J_C^0$ that $\alpha$ induces.  Because the locus of indeterminacy is covered by rational curves and a Jacobian variety cannot contain any rational curve,  $\beta$ is in fact an automorphism of $J_C^0$. Furthermore, $\beta$ must preserve the principal polarization because the N\'eron--Severi group of $J^0_C$ is cyclic for a very general $C$ by \cite[Corollary 17.5.2]{bl} and because $\Theta_C$ is the unique generator of the  N\'eron--Severi group of $J^0_C$ that is ample.  We conclude using a version of the Torelli theorem \cite[Theorem 12.1]{milne} that implies that $\beta$ must lie in the group generated by translations and the involution $L\mapsto L^{-1}$.  By Fact~\ref{franchetta}, this group  is the group of automorphisms of the form \eqref{bir} with $k(2-2g) + \sum d_j=0$.  Since $\alpha$ coincides with an automorphism of the form \eqref{bir} on the generic fiber, it must be equal to that automorphism.
\end{proof}

In Corollary~\ref{corfranchetta} it is essential to assume that $\alpha$ commutes with the forgetful maps. The problem of characterizing arbitrary birational maps $\J_{g,n}^{e_1} \dashrightarrow \J_{g,n}^{e_2}$ is harder than the problem of classifying birational maps $\mathcal{M}_{g,n} \dashrightarrow \mathcal{M}_{g,n}$, and this second classification is not available, even for $g$ and $n$ large.

\subsection{Extensions of  Abel--Jacobi sections}

\label{Section: AbelJacobi}

Motivated by Fact~\ref{franchetta}, in this section we  fix integers $(k; d_1, \ldots, d_n)$, we let $d:= k(2-2g) + d_1 + \ldots +d_n$, and apply the earlier results of this paper to analyze extensions to $\Mb_{g,n}$ of the Abel--Jacobi section $\sigma=\sigma_{k, \vec{d}}$ defined by Rule~\eqref{Eqn: DefOfAbelSection}.  The main result is  Corollary~\ref{extensionAJ}, in which we characterize the nondegenerate $\phi$'s of the stability space $V_{g,n}^d$  such that the Abel--Jacobi section extends to a well-defined morphism $\sigma \colon \Mb_{g,n} \to \Jb_{g,n}(\phi)$.

Corollary~\ref{extensionAJ} follows from the more general Corollary~\ref{Cor: AbelMapIndeterminacy}, which describes the locus of indeterminacy of $\sigma \colon \Mb_{g,n} \dashrightarrow \Jb_{g,n}(\phi)$ as the closure of the locus of pointed curves $(C, p_1, \ldots, p_n)$ with $2$ smooth irreducible components meeting in at least $2$ nodes such that   $\omega_C^{\otimes -k} \otimes \mathcal{O}_C(d_1 p_1 + \ldots + d_n p_n)$ fails to be $\phi(C,p_1, \ldots, p_n)$-stable. To prove this result, we first observe that all nondegenerate $\phi \in V_{g,n}^d$ that have the same projection (under the isomorphism of Corollary~\ref{isoalphan}) to $D_{g,n}$ correspond to isomorphic moduli stacks $\Jb_{g,n}(\phi)$, so we can reduce to the case where $\omega_C^{\otimes -k} \otimes \mathcal{O}_C(d_1 p_1 + \ldots + d_n p_n)$ is $\phi$-stable on all curves with at most $1$ node. The proof that the indeterminacy locus is not smaller than the one that we claimed essentially follows from the fact that there exists a unique rank~$1$ torsion free sheaf that extends to $\Cb_{g,n}$ the restriction of $\omega_C^{\otimes -k} \otimes \mathcal{O}_C(d_1 p_1 + \ldots + d_n p_n)$ to $\Cb_{g,n}^{\leq 1}$, defined as the universal curve over $\Mb_{g,n}^{\leq 1}$ (the moduli stack of stable curves with at most $1$ node).  The problem of resolving the indeterminacy of the Abel--Jacobi sections  was raised  by Grushevsky--Zakharov in \cite{grushevsky}, and in Remark~\ref{Remark: RelationWIthGrushevsky}, we discuss how that work relates to the present paper.

In analyzing the locus of indeterminacy, the following line bundles on the universal curve play a fundamental role.

\begin{definition} \label{defOD}
Let $\OD$ be the line bundle on the universal curve $\Cb_{g,n}$ defined by
\[
\OD:= \omega_{\pi} ^{\otimes -k} \otimes \mathcal{O}_C( d_1 \Sigma_1 + \ldots + d_n \Sigma_n),
\]
and let $\phi_{k, \vec{d}} \in V^{d}_{g, n}$ be its multidegree:
	\begin{equation} \label{phid}
		\phi_{k,\vec{d}}(\Gamma(i,S)) := (d_S +(2i-1) k, d- d_S - (2i-1) k), \quad \phi_{k, \vec{d}}(\Gamma_j):= (d_j, d-d_j).
	\end{equation}
(By Corollary~\ref{isoalphan}, the Equations~\eqref{phid} define a unique element of $V_{g,n}^d$).

For $\phi \in V_{g,n}^d$ nondegenerate, we define the following modification of $\OD$ (slightly generalizing what we did for $k=0$  in \cite[Section 5]{kasspa}):
	\[
		\ODF := \OD \otimes \mathcal{O}_C \left(\sum_{(i, S)} \left( -d_{S} - k (2 i -1) + \left\lfloor \phi(\Gamma(i, S))(v) + \frac{1}{2} \right\rfloor \right) \cdot {C}^-_{i, S} \right),
	\]
where $v$ is the first vertex of $\Gamma(i,S)$ according to the convention we fixed in Section~\ref{notationmoduli}, and $C^-_{i,S}$ the choice of a component that we fixed in Section~\ref{Picard}.

\end{definition}
The line bundle $\ODF$ is defined so that its restriction to smooth pointed curves equals the restriction of $\OD$, and its restriction to stable curves with at most $1$ node is $\phi$-stable. Stability follows from the following proposition, which describes the properties of $\ODF$ that we will use next.

\begin{proposition} \label{Prop: StabilityVectorForAbelMap}
The line bundle $\ODF$ satisfies the following.
\begin{enumerate}
\item The restriction $\ODF|_{(C,p_i)}$ to a stable pointed curve $(C,p_i)$ with $1$ node  is  $\phi(\Gamma_{(C,p_i)})$-stable.
\item The restriction of $\ODF$ to a stable pointed curve with $2$ smooth components and at least $2$ nodes equals the restriction of $\OD$.
\end{enumerate}
\end{proposition}
\begin{proof}
The proof of the first claim follows by computing the bidegrees of $\ODF$ on all stable pointed curves with $2$ components and $1$ separating node. The second claim follows from the fact that  the line bundles $\mathcal{O}_C({C}^-_{i, S})$ become trivial when restricted to curves with $2$ smooth irreducible components and at least $2$ nodes.
\end{proof}

We can now state and prove our first characterization of the indeterminacy locus of the Abel--Jacobi section $\sigma=\sigma_{k,\vec{d}}$.
\begin{proposition} \label{Prop: AbelMapIndeterminacy}
	Given a nondegenerate $\phi \in V_{g,n}^d$, the locus of indeterminacy of the rational map $\sigma \colon \Mb_{g, n} \dashrightarrow \Jb_{g, n}(\phi)$ defined by \eqref{Eqn: DefOfAbelSection} is the locus $Z(\phi)$ of stable curves $(C, p_1, \ldots, p_n)$  such that the restriction $\ODF|_{(C, p_i)}$ fails to be $\phi(\Gamma_{(C, p_i)})$-stable.
\end{proposition}
\begin{proof}
Define $U(\phi)$ to be the locus of determinacy of $\sigma$ (i.e.~the locus  where $\sigma$ is a well-defined morphism to $\Jb_{g,n}(\phi)$). Extending $\sigma$ by the rule
  \[(C,p_i) \mapsto \ODF|_{(C, p_i)},\] we deduce the inclusion  \begin{equation}  \Mb_{g,n} \setminus {Z}(\phi) \subseteq U(\phi).\label{inclusion} \end{equation} To conclude, we need to show that \eqref{inclusion} is an equality.

%   From the first part of Proposition~\ref{Prop: StabilityVectorForAbelMap} w

We first consider the case $n=0$. When $g>2$ and $n=0$ there is no nondegenerate $\phi$ in degree $d$ for $d$ the total degree of $\ODF$. Indeed, when $n=0$ we have that $d$ is a multiple of $2g-2$, and with the hypothesis that $g>2$, we deduce $\gcd(d+1-g, 2g-2) >1$. In Remark~\ref{phicandegenerate} we observed that there is no nondegenerate $\phi$ in $V_{g,0}^d$ when $\gcd(d+1-g, 2g-2) >1$.  The special case when $g=2$ and $n=0$ is easily dealt with by observing that $Z(\phi) = \emptyset$, so that the reverse inclusion of \eqref{inclusion} holds trivially.

From now on we assume $n \geq 1$. Under this assumption, we know by \cite[Lemma 3.35]{kasspa} that there exists a tautological sheaf $F_{\text{tau}}(\phi)$ on $\Jb_{g,n}(\phi) \times_{\Mb_{g,n}} \Cb_{g,n}$. Denote by $\tilde{\sigma}$ the rational map $\Cb_{g,n} \dashrightarrow \Jb_{g,n}(\phi) \times_{\Mb_{g,n}} \Cb_{g,n}$ obtained by pulling back $\sigma$.  %the section $\sigma \colon \Mb_{g,n} \to \Jb_{g,n}(\phi)$.

By the definition of a tautological sheaf and of $\tilde{\sigma}$, we have an isomorphism of line bundles \begin{equation} \label{equals} \tilde{\sigma}^*|_{\pi^{-1}\left(\Mb_{g,n} \setminus {Z}(\phi)\right)} ( F_{\text{tau}}(\phi)) \cong \ODF|_{\pi^{-1}\left(\Mb_{g,n} \setminus {Z}(\phi)\right)}.\end{equation} By the first part of Proposition~\ref{Prop: StabilityVectorForAbelMap}, the complement of $\pi^{-1}\left(\Mb_{g,n} \setminus {Z}(\phi)\right)$ has codimension at least $2$ in $\Cb_{g,n}$, hence so does $\pi^{-1}\left(U(\phi)\right)$ by Inclusion~\eqref{inclusion}. We conclude that the isomorphism in \eqref{equals} extends over $U(\phi)$ using Corollary~\ref{Cor: determinedcod2} (which shows quite generally that, on a smooth Deligne--Mumford stack, an isomorphism of families of rank~$1$ torsion-free sheaves extends over a codimension $2$ locus). %Because $\ODF$ on $\Cb_{g,n}$ is one such family (in fact, it is even a line bundle), we deduce that if $\tilde{\sigma}$ is defined at $\pi^{-1}(C, p_1, \ldots, p_n)$, then we have the isomorphism of line bundles
%\begin{equation}  \label{equals2} \tilde{\sigma}^*|_{\pi^{-1}(C, p_1, \ldots, p_n)} ( F_{\text{tau}}(\phi)) \cong \ODF|_{\pi^{-1}(C, p_1, \ldots, p_n)}.\end{equation}

 We are now ready to prove the reverse inclusion of \eqref{inclusion}. Assume $(C,p_1, \ldots, p_n) \in U(\phi)$. By restricting Isomorphism~\eqref{equals} to $\pi^{-1}(C,p_1, \ldots, p_n)$ we obtain \begin{equation}  \label{equals2} \tilde{\sigma}^*|_{\pi^{-1}(C, p_1, \ldots, p_n)} ( F_{\text{tau}}(\phi)) \cong \ODF|_{\pi^{-1}(C, p_1, \ldots, p_n)}.\end{equation}  Because the left-hand side of \eqref{equals2} is  $\phi(\Gamma_{(C, p_i)})$-stable (by definition of a tautological sheaf), so is the right-hand side. We conclude that $(C, p_1, \ldots, p_n) \notin  Z(\phi)$, and this concludes our proof that $\Mb_{g,n} \setminus {Z}(\phi) = U(\phi)$.
\end{proof}

%We observe that the proof of Proposition~\ref{Prop: AbelMapIndeterminacy} implies that the indeterminacy locus of $\sigma \colon \Mb_{g,n} \dashrightarrow \Jb_{g,n}(\phi)$ does not change if the target $\Jb_{g,n}(\phi)$ is replaced with its open substack that parameterizes line bundles.

While $\phi$-stability of $\OD$ gives a \emph{sufficient} condition for the corresponding Abel--Jacobi section to extend to a regular section $\Mb_{g,n}\to \Jb_{g,n}(\phi)$, Proposition~\ref{Prop: AbelMapIndeterminacy} implies that $\phi$-stability of $\ODF$ is an \emph{equivalent} condition for the same extension problem.

We now rewrite the statement of Proposition~\ref{Prop: AbelMapIndeterminacy} purely in terms of the original line bundle $\OD$. %To do so, we pay the price of partly loosing the modular interpretation of the indeterminacy locus of $\sigma$ (the new statement will involve taking the Zariski closure). %We let
%\[
%Z(\phi) := \{ (C, p_1, \ldots, p_n) \in \Mb_{g,n} : \ C \text{ has two smooth irreducible components and two nodes, and } \OD|_{C, p_i} \text{ fails to be } \phi(\Gamma_{C, p_i}) \text{-stable} \}.
%\]
\begin{corollary} \label{Cor: AbelMapIndeterminacy} Given a nondegenerate $\phi \in V_{g,n}^d$, the locus of indeterminacy of the rational map $\sigma \colon \Mb_{g, n} \dashrightarrow \Jb_{g, n}(\phi)$ defined by \eqref{Eqn: DefOfAbelSection} is the closure of the locus $T(\phi)$ of stable curves $(C, p_1, \ldots, p_n)$ with $2$ smooth irreducible components and at least $2$ nodes, such that  $\OD|_{(C, p_i)}$  fails to be  $\phi(\Gamma_{(C, p_i)})$-stable.
\end{corollary}

\begin{proof}
By applying Proposition~\ref{Prop: AbelMapIndeterminacy}, the claim is reduced to proving that $Z(\phi)$ equals the closure of ${T}(\phi)$ in $\Mb_{g,n}$.

Let $\mathcal{T}$ be the locally closed locus of $\Mb_{g,n}$ consisting of stable curves with $2$ smooth irreducible components and at least $2$ nodes. By the second part of Proposition~\ref{Prop: StabilityVectorForAbelMap}, we have that $Z(\phi) \cap \mathcal{T} = T(\phi)$. We deduce the equality $Z(\phi) = \overline{T}(\phi)$ from the fact that $Z(\phi) \subseteq \overline{\mathcal{T}}$ (by the first part of Proposition~\ref{Prop: StabilityVectorForAbelMap}) and from the fact that $Z(\phi)$ is closed (by its very definition). %This concludes our proof.
\end{proof}

We conclude by giving explicit characterizations of the nondegenerate parameters $\phi \in V_{g,n}^d$ such that $\OD$ is $\phi$-stable, and of the nondegenerate parameters $\phi$ such that the corresponding Abel--Jacobi section  extends to a regular section $\Mb_{g,n}\to \Jb_{g,n}(\phi)$. %$\ODF$ is $\phi$-stable.

We begin by observing that the line bundle $\OD$ is $\phi$-stable for $\phi = \phi_{k, \vec{d}}$, simply because the latter stability parameter is by definition the multidegree of $\OD$. The parameter $\phi_{k, \vec{d}}$ is degenerate (so there is not an associated Deligne--Mumford stack $\Jb_{g,n}(\phi_{k, \vec{d}})$), but $\phi$-stability of $\OD$ is preserved within the polytope centered at $\phi_{k, \vec{d}}$ that we are now going to define.

\begin{definition} \label{polytopeQ} Define the polytope $\mathcal{Q}(\phi_{k, \vec{d}})$ in $V_{g,n}^d$ as the collection of all $\phi \in V_{g,n}^d$ that satisfy the inequalities %through the isomorphism of Corollary~\ref{isoalphan} by the  inequalities
\begin{equation} \label{Qdgn}
\left|\phi(\Gamma)(v) - \phi_{k, \vec{d}}(\Gamma)(v) \right| < \frac{\alpha}{2}
\end{equation}
for all loopless graphs $\Gamma$ with $2$ vertices $v$ and $w$ of genera $i$ and $g-\alpha+1-i$ respectively, joined by $\alpha$ edges and with the $S\subseteq [n]$ markings on the first vertex.
\end{definition}

 We remark that $\mathcal{Q}(\phi_{k, \vec{d}})$ is \emph{not} a stability polytope in the sense of Definition~\ref{Definition: PolytopesFamilyCurves}. We introduced the polytope $\mathcal{Q}(\phi_{k, \vec{d}})$ to formulate the following result.

\begin{corollary} \label{Cor: StabilityVectorForAbelMap}
	For $\phi\in V_{g,n}^d$ nondegenerate, the line bundle $\OD$ is $\phi$-stable if and only if $\phi$ belongs to the polytope ${\mathcal{Q}}(\phi_{k, \vec{d}})$ of Definition~\ref{polytopeQ}.
\end{corollary}
\begin{proof} This follows from Lemma~\ref{Lemma: PhiDeterminedByTwoComponents} (the fact that $\phi$-stability can be checked on curves with $2$ smooth irreducible components) and from Definition~\ref{phistab} (the definition of $\phi$-stability). % (Definitionexplicit $\phi$-stability inequalities on such curves (see Example~\ref{exampledollarsign}). %,  Definition~\ref{phistab} (the definition of $\phi$-stability) and the fact that, for a nondegenerate parameter, a sheaf is semistable if and only if it is stable.
 \end{proof}
We deduce the following characterization of the set of stability parameters such that the Abel--Jacobi section extends.
\begin{corollary} \label{extensionAJ}
	For $\phi \in V_{g,n}^d$ nondegenerate, the section $\sigma$ extends to a regular section $\Mb_{g,n} \to \Jb_{g, n}(\phi)$ if and only if, under the isomorphism of Corollary~\ref{isoalphan}, the projection  of the polytope ${\mathcal{Q}}(\phi_{k, \vec{d}})$ to $D_{g, n}$ contains the projection of $\phi$.
\end{corollary}
\begin{proof}
	By Corollary~\ref{Cor: AbelMapIndeterminacy} the Abel--Jacobi section extends to a regular morphism if and only if the restriction $\OD|_{(C, p_i)}$  to all curves $(C, p_i)$ with $2$ smooth irreducible components and at least $2$ nodes is  $\phi(\Gamma_{(C, p_i)})$-stable. By applying Theorem~\ref{stabilityspace} we deduce that this is equivalent to the fact that the projection  of ${\mathcal{Q}}(\phi_{k, \vec{d}})$ to $D_{g, n}$ contains the projection of $\phi$.  % $\OD, Corollary~\ref{Cor: StabilityVectorForAbelMap}, and Theorem~\ref{stabilityspace}.
\end{proof}
Concretely, this means that $\sigma$ extends to a regular section $\Mb_{g,n} \to \Jb_{g,n}(\phi)$ if and only if  Equation~\eqref{Qdgn} is satisfied by $\phi$ for all loopless graphs $\Gamma$ with $2$ vertices and at least $\alpha \geq 2$ edges.

We conclude this section by comparing the result just proven with the work of Grushevsky--Zakharov in \cite{grushevsky}.
\begin{remark}\label{Remark: RelationWIthGrushevsky}
	 For $\vec{d}$ satisfying $\sum d_i =0$, Grushevsky--Zakharov describe the indeterminacy of the Abel--Jacobi section $\sigma_{\vec{d}} := \sigma_{0, \vec{d}}$ considered as a morphism into a stack $\mathcal{X}'_{g} \to \mathcal{A}'_{g}$ they call Mumford's partial compactification.  This partial compactification is an extension of the universal family of principally polarized abelian varieties $\mathcal{X}_{g} \to \mathcal{A}_{g}$ that is constructed so that the fiber over a point of $\mathcal{A}'_{g} - \mathcal{A}_{g}$ is an explicit compactification of a semiabelian variety with $1$-dimensional maximal torus called a rank~$1$ degeneration.

	The Torelli map extends to a regular morphism $\Mb_{g,0}^{t \le 1} \to \mathcal{A}'_{g}$ out of the locus $\Mb_{g,0}^{t \le 1} \subseteq \Mb_{g,0}$ of  curves whose generalized Jacobian has torus rank at most $1$.  The pullback of  $\mathcal{X}'_{g}$ under the composition  $\Mb^{t \le 1}_{g,n} \to \Mb^{t \le 1}_{g,0} \to \mathcal{A}_{g}'$ is an extension of the universal Jacobian in degree $0$ that we will denote by $\mathcal{Y}'_{g,n}$ or by $\mathcal{Y}'_{g}$ when $n=0$.
	
	The extension $\mathcal{Y}'_{g,n}$ is not equal to a compactified universal Jacobian $\Jb_{g,n}(\phi)$ associated to a nondegenerate $\phi$.  Indeed, for all such $\phi$, the compactified Jacobian $\overline{J}_{C}(\phi)$ of a stable pointed curve $C$ that is the union of $2$ smooth curves meeting in $2$ nodes is a reducible variety (with two components corresponding to the two $\phi$-stable bidegrees), while the analogous fiber of $\mathcal{Y}'_{g,n}$ is irreducible.
	
	The family $\mathcal{Y}'_{g, n}$ is, however,  related to  compactified universal Jacobians associated to the degenerate parameter $\phi_{\text{can}}^{0}$, although the relation is somewhat subtle.  The stack $\Jb_{g,n}( \phi^{0}_{\text{can}})$ admits a good moduli scheme $\overline{P}_{g}^{0}$ (the scheme constructed in \cite{caporaso, panda}), and the authors believe it is expected by  experts that, over the automorphism-free locus,  $\mathcal{Y}'_{g}$ is isomorphic to $\overline{P}_{g}^{0}$.  For example, Caporaso describes the fiber $\overline{P}_{C}$ of $\overline{P}^{0}_{g} \to \Mb_{g,0}$ over a curve  $C$ with $2$ nonseparating nodes as a rank~$1$ degeneration in \cite[Figure~8]{caporaso}, and on \cite[page~240]{namikawa}, Namikawa indicates a relation between families over $\mathcal{A}'_{g}$, and more generally toroidal compactifications of $\mathcal{A}_{g}$,  and Oda--Seshadri's compactified Jacobians.  A proof that the two families are isomorphic does not, however,  seem to be available, although Alexeev has proven a parallel statement for $\overline{P}^{g-1}_{g}$ \cite[Corollary~5.4]{alexeev04}.

\end{remark}

\begin{example}	
	Grushevsky--Zakharov's main result \cite[Example~6.1, 6.2]{grushevsky} about the locus of indeterminacy of $\sigma_{\vec{d}} \colon \Mb_{g,n} \dashrightarrow \mathcal{Y}'_{g,n}$  is that the locus of indeterminacy equals the closure of the locus of stable marked curves with $2$ nonseparating nodes except for certain special choices of $\vec{d}$.  As we explained in Remark~\ref{Remark: RelationWIthGrushevsky}, $\mathcal{Y}'_{g,n}$ should be compared with $\Jb_{g,n}(\phi^{0}_{\text{can}})$ and the related scheme $\overline{P}^{0}_{g}$.  The main results of this section do not apply to $\Jb_{g,n}(\phi^{0}_{\text{can}})$ because $\phi^{0}_{\text{can}}$ is degenerate, but Grushevsky--Zakharov's result is consistent with our results, in a manner that we explain below.  Their proof of the result is, however, different from the proof we give, and we explain this as well.
	
	In \cite[Example~6.1]{grushevsky}, Grushevsky--Zakharov show that, except for certain special $\vec{d}$, the locus of indeterminacy contains every stable marked curve  $(C, p_1, \dots, p_n)$ that is the union of two general smooth curves $C_1$, $C_2$ of positive genera  meeting in two nodes and marked so that only $p_1$ lies on  $C_1$.  Let us analyze the behavior of $\sigma_{\vec{d}}$ around $(C, p_1, \dots, p_n)$ using the ideas of this section under the simplifying assumption that $C$ is automorphism-free.

	In this section, we  analyzed the indeterminacy of $\sigma_{\vec{d}} \colon \Mb_{g,n} \dashrightarrow \Jb_{g,n}(\phi)$ for $\phi$ nondegenerate by relating it to the stability of a sheaf on the universal family of curves.  That analysis shows  $(C, p_1, \dots,p_n)$ lies in the locus of indeterminacy when $\phi=\phi_{\epsilon}$ is nondegenerate and sufficiently close to $\phi_{\text{can}}^{0}$.  Assuming the result relating $\mathcal{Y}'_{g}$ to $\overline{P}^{0}_{g}$, we can alternatively deduce this from \cite{grushevsky}.  Indeed, there is a tautological morphism $\Jb_{g,n}(\phi_{\epsilon}) \to \Jb_{g,n}(\phi^{0}_{\text{can}})$ since every $\phi_{\epsilon}$-stable sheaf is $\phi_{\text{can}}^{0}$-semistable. Composing this map with the natural morphisms $\Jb_{g,n}(\phi^{0}_{\text{can}}) \to \Jb_{g,0}(\phi^{0}_{\text{can}})$ and $\Jb_{g,0}(\phi_{\text{can}}^{0}) \to \overline{P}^{0}_{g}$, we obtain a morphism $\Jb_{g,n}(\phi_{\epsilon}) \to \overline{P}^{0}_{g}$ that factors $\sigma_{\vec{d}} \colon \Mb_{g,n} \dashrightarrow \overline{P}^{0}_{g}$.  From the existence of the factorization, we can deduce the indeterminacy of the morphism into $\Jb_{g,n}(\phi_{\epsilon})$  from the  indeterminacy of the morphism into $\overline{P}^{0}_{g}$.
	
	While our result can be deduced from loc.~cit. (assuming a comparison result relating $\overline{P}^{0}_{g}$ to $\mathcal{Y}'_{g}$), the argument given there is different. Rather than analyzing stability conditions,  the authors of loc.~cit.~analyze the restriction of $\sigma_{\vec{d}}$ to certain curves.  Specifically, they construct some explicit morphisms $f \colon B \to \Mb_{g,n}$ out of a smooth curve $B$ that send a distinguished point $0 \in B$ to $(C, p_1, \dots, p_n)$.  By the valuative criterion, the composition $\sigma_{\vec{d}} \circ f \colon B \to \Mb_{g,n} \dashrightarrow \mathcal{Y}'_{g}$  extends to a regular morphism.  If $\sigma_{\vec{d}}$ had no indeterminacy around $(C, p_1, \dots, p_n)$, the image $(\sigma_{\vec{d}} \circ f)(0)$ would be independent of $f$, but they show by direct computation that this is not the case.

	Their result, or rather the analogue for the map into $\Jb_{g,n}(\phi_{\epsilon})$, can be explained in the language of this paper as follows.  The morphism $f$ corresponds to a family $\mathcal{C} \to B$ of stable marked curves with a distinguished point $ 0 \in B$ mapping to  $(C, p_1, \dots, p_n) \in \Mb_{g,n}$, and the morphism $\sigma_{\vec{d}} \circ f \colon B \to \Jb_{g,n}(\phi_{\epsilon})$ corresponds to a family of $\phi_{\epsilon}$-semistable sheaves that extends the family $\mathcal{O}( d_1 p_1 + \dots d_n p_n)$ on the generic fiber.  Such a family can be constructed explicitly using twistor sheaves (i.e.~sheaves on $\mathcal{C}$ that restrict to the trivial line bundle on the generic fiber), and except in the special cases,  changing $\mathcal{C} \to B$ changes the sheaf.  Twistor sheaves for some families similar to $\mathcal{C} \to B$ are explicitly constructed and described in \cite[Section~2.2]{marcus}.  % The same is true if we replace $\phi_{\text{can}}$ and $\overline{P}_{g}^{0}$ with $\phi_{\epsilon}$ and $\Jb_{g,n}(\phi_{\epsilon})$.
	
	The point $(\sigma_{\vec{d}} \circ f)(0)$ becomes independent of the choice of $f$ if we pass from $\phi_{\epsilon}$ to a nondegenerate parameter $\phi'$ with the property that $\mathcal{O}( d_1 p_1 + \dots + d_n p_n)$ is $\phi'$-stable.  Indeed,  $(\sigma_{\vec{d}} \circ f)(0)$ is then just the restriction of $\mathcal{O}( d_1 p_1 + \dots + d_n p_n)$. While there is a straightforward  description of what is happening in terms of stability conditions, the geometry is subtle.  The manner in which the geometry changes when  passing from $\overline{P}^{0}_{g}$ to $\Jb_{g,n}(\phi_{\epsilon})$ can be described explicitly: the pullback of $\overline{P}^{0}_{g}$ to $\Mb_{g,n}$ is blown-up so that the fiber over  $(C, p_1, \dots, p_n)$ changes from the variety labeled ``Special case" in \cite[Figure~8]{caporaso} to the variety labeled ``General case".  It is more difficult to describe how the geometry changes in passing from $\Jb_{g,n}(\phi_{\epsilon})$ to $\Jb_{g,n}(\phi')$.  Locally around $(C, p_1, \dots, p_n)$, both families have the same fibers, so the manner in which the indeterminacy is resolved cannot be seen by  analyzing fibers.
\end{example}

\subsection{Different fine compactified universal Jacobians}

\label{Section: different}

The main goal of this section is to show, for fixed $(g,n)$, the existence of non-isomorphic fine compactified universal Jacobians $\Jb_{g, n}(\phi)$.  We do this by taking advantage of the natural action of a certain group $\widetilde{\pr}_{g,n}$  (defined in Definition~\ref{Def: GroupAction}) on the set of stability polytopes $\mathcal{P}_{g,n}$ (defined in Definition~\ref{Definition: PolytopesFamilyCurves}). In Lemma~\ref{isoiffsameorbit} we prove that  two compactified Jacobians $\Jb_{g,n}(\phi_1)$ and $\Jb_{g,n}(\phi_2)$ are isomorphic over $\Mb_{g,n}$ if and only if $\mathcal{P}(\phi_1)$ lies in the same orbit as $\mathcal{P}(\phi_{2})$.

We study the property of this group action in Corollary~\ref{whentransitive}, where we show that it fails to be transitive except for few, low genus, special cases of $(g,n)$. We immediately deduce that, except for the special cases, for a given $(g,n)$, there exist at least two $\Jb_{g,n}(\phi)$'s that are not isomorphic over $\Mb_{g,n}$.  This is Corollary~\ref{noisocommuting}, and in  Corollary~\ref{nonisomorphic}, we prove the stronger statement that, provided $\Mb_{g,n}$ is of general type, there exist $\Jb_{g,n}(\phi)$'s that are not isomorphic as stacks (rather than as stacks over $\Mb_{g,n}$).

A very similar problem is to fix $d \in \mathbb{Z}$ and classify fine compactified universal Jacobians $\Jb_{g, n}(\phi)$ of degree $d$. We observe in Remark~\ref{fixedd} that this classification problem is very similar to the previously described one, where the degree of the fine compactified universal Jacobians is allowed to vary. % Corollaries~\ref{noisocommuting} and~\ref{nonisomorphic} remain the same for fixed $d$ provided $n\geq 1$, and that the problem is trivial when $n=0$.

The group acting on stability polytopes is the following one.

\begin{definition} \label{Def: GroupAction}
	Let $\widetilde{\pr}_{g,n}$ be the generalized dihedral group defined by the action $L \mapsto L^{(-1)^t}$ of $t \in \mathbb{Z}/2\mathbb{Z}$ on $\picrel_{g,n}(\mathbb{Z})$.  In other words, $\widetilde{\pr}_{g,n}$ is the semi-direct product  \[\widetilde{\pr}_{g,n}:= \left(\picrel_{g,n}(\mathbb{Z})\right) \rtimes \mathbb{Z}/2 \mathbb{Z}.\]
\end{definition}

The group $\widetilde{\pr}_{g,n}$ acts on families of rank~$1$ torsion-free sheaves on families of stable pointed curves by
the rule
\begin{equation} \label{psi}
\psi(L,t) \colon F \mapsto F^{(-1)^t} \otimes L.
\end{equation}
(That $F \mapsto F^{-1}$ gives a well-defined map in families follows from Lemma~\ref{basechange} in Section~\ref{Section: Appendix}).
Similarly, the group $\widetilde{\pr}_{g,n}$ also acts on the stability space $V_{g,n}$ by the affine endomorphisms
\begin{equation} \label{lambda}
\lambda(L,t) \colon \phi \mapsto (-1)^{t} \cdot \phi + \deg(L).
\end{equation}

The following key observation relates the two actions $\psi$ and $\lambda$.
\begin{lemma} \label{key} Assume that $F$ has degree $d$ and that $\phi \in V_{g,n}^d$. Then $F$ is $\phi$-(semi)stable if and only if $\psi(L,t)(F)$ is $\lambda(L,t)(\phi)$-(semi)stable.
 In particular, when $\phi$ is nondegenerate, $\psi(L,t)$ induces a well-defined isomorphism $\Jb_{g,n}(\phi) \to \Jb_{g,n}(\lambda(\phi))$ that commutes with the forgetful maps to $\Mb_{g,n}$. \end{lemma}
\begin{proof} The claim follows immediately from  Definition~\ref{phistab} and Definition~\ref{Def: moduliStacks}.
\end{proof}

In fact, all isomorphisms that commute with the forgetful maps are defined by Rule~\eqref{psi} for a suitable choice of $(L,t) \in \widetilde{\pr}_{g,n}$, as we prove in the next lemma.

\begin{lemma} \label{isoiffsameorbit} Let $\phi_1 \in V_{g,n}^{e_1}$ and $\phi_2 \in V_{g,n}^{e_2}$ be nondegenerate, and assume $\alpha \colon \Jb_{g,n}(\phi_1) \to \Jb_{g,n}(\phi_2)$ is a birational morphism that commutes with the forgetful maps to $\Mb_{g,n}$. Then there exists $(L,t) \in \widetilde{\pr}_{g,n}$ such that $\alpha = \psi(L, t)$.
\end{lemma}

\begin{proof}
By Corollary~\ref{corfranchetta}, the restriction of $\alpha$ to $\J_{g,n}^{e_1}$ is an isomorphism given by
\[
 L \mapsto L^{(-1)^t} \otimes  \omega_{\pi}^{\otimes - k} \otimes \mathcal{O}_C( d_1 \Sigma_1+ \ldots + d_n \Sigma_n)
\]
for some $(k; d_1, \ldots, d_n) \in \mathbb{Z}^{n+1}$ and $t \in \{0,1\}$. The latter can be extended to a well-defined morphism $\Jb_{g,n} ^{\leq1}(\phi_1) \to \Jb_{g,n} ^{\leq1}(\phi_2)$ (here $\Jb_{g,n}^{\leq 1}(\phi)$ denotes the restriction of $\Jb_{g,n}(\phi)$ to $\Mb_{g,n}^{\leq 1}$, the substack parameterizing curves with at most $1$ node) by the rule $\psi(L,t)$ as defined by Equation~\eqref{psi}. Here we have defined
\[
L:=  \omega_{\pi}^{\otimes -k} \otimes \mathcal{O}_C( d_1 \Sigma_1+ \ldots + d_n \Sigma_n + \sum a_{i,S} \cdot {C}^-_{i,S})
\]
for $C_{i,S}^-$ as defined in Section~\ref{Picard}, and $a_{i,S}$ defined to be the componentwise approximation to the nearest integer of the restriction of \[\phi_2 - (-1)^t \phi_1 - \deg \left(  \omega_{\pi}^{\otimes -k} + d_1 \Sigma_1+ \ldots + d_n \Sigma_n \right)\] to a general curve of $\Delta(i,S)$.

While it is a priori not clear that this $\psi(L,t)$ extends to a morphism $\Jb_{g,n}(\phi_1) \to \Jb_{g,n}(\phi_2)$, the argument in the above paragraph shows that the restriction of $\alpha$ to $\Jb_{g,n} ^{\leq1}(\phi_1)$ coincides with $\psi(L,t)$. To conclude, we need to prove that $\alpha$ and $\psi(L,t)$  coincide on $\Jb_{g,n}(\phi_1)$.

When $n \geq 1$, consider a tautological sheaf $F_{\text{tau}}(\phi_2)$ on $\Jb_{g,n}(\phi_2) \times_{\Mb_{g,n}} \Cb_{g,n}$ (which exists by \cite[Lemma 3.35]{kasspa}). The pullback via $\alpha \times \text{Id}$ and via $\psi \times \text{Id}$ of $F_{\text{tau}}(\phi_2)$ coincide on the locus \begin{equation} \label{locus} \Jb_{g,n}(\phi_1)^{\leq 1} \times_{\Mb_{g,n}^{\leq 1}} \Cb_{g,n}^{\leq 1}\end{equation} where the underlying curve has at most $1$ node. Because the locus \eqref{locus} is an open substack of $\Jb_{g,n}(\phi) \times_{\Mb_{g,n}} \Cb_{g,n}$ whose complement has codimension $2$, by Lemma~\ref{Lemma: G1S2} and Corollary~\ref{Cor: determinedcod2} the two pullbacks must coincide everywhere. This implies that $\psi(L,t) \times \text{Id}$ can be extended to coincide with $\alpha \times \text{Id}$ everywhere, which implies that $\alpha=\psi(L,t)$ on $\Jb_{g,n}(\phi_1)$.

When $n=0$ apply the same argument of the above paragraph after first passing to an \'etale cover $\mathcal{U} \to \Jb_{g,0}(\phi)$ such that a tautological sheaf exists on  $\mathcal{U} \times_{\Mb_{g,0}} \Cb_{g,0}$. (To prove that $\alpha \times \text{Id}$ and  $\psi \times \text{Id}$ coincide, it is enough to check that the same holds \'etale locally.) This concludes the proof.
\end{proof}

Here is a corollary of Lemma~\ref{isoiffsameorbit} that gives a new significance to the polytope decomposition of $D_{g,n}$ (see Theorem~\ref{stabilityspace} and, in particular, the set  of equations in~\eqref{otherwalls}).
\begin{corollary} Fix $d \in \mathbb{Z}$ and let $\phi_1, \phi_2 \in V_{g,n}^d$ be nondegenerate. Then there exists an isomorphism (or, equivalently, a birational morphism) $\Jb_{g,n}(\phi_1) \to  \Jb_{g,n}(\phi_2)$ that extends the identity on $\mathcal{J}_{g,n}^d$ and that commutes with the forgetful maps to $\Mb_{g,n}$ if and only if the projections of the stability polytopes $\mathcal{P}(\phi_1)$ and $\mathcal{P}(\phi_2)$ to $D_{g,n}$ coincide.
\end{corollary}

\begin{proof}
Combine Lemma~\ref{key} and Lemma~\ref{isoiffsameorbit}. \end{proof}

Each affine endomorphism $\lambda(L,t)$ maps stability polytopes (from Definition~\ref{Definition: PolytopesFamilyCurves}) to stability polytopes. Thus $\lambda$ induces an action, that we call $\mu$, of $\widetilde{\pr}_{g,n}$ on the set of stability polytopes $\mathcal{P}_{g,n}$.  Lemma~\ref{isoiffsameorbit} gives a combinatorial answer to a natural classification problem, as it implies that the isomorphism classes over $\Mb_{g,n}$ of $\Jb_{g,n}(\phi)$ for $\phi$ varying in $\coprod_{d \in \mathbb{Z}} V_{g,n}^d$ and nondegenerate,  are in natural bijection with the quotient set  $\mathcal{P}_{g,n}/ \widetilde{\pr}_{g,n}$.

 %uadmit an isomorphism over $\Mb_{g,n}$ to studying when the stability polytope $\mathcal{P}(\phi_1)$ lies in the same orbit of $\mathcal{P}(\phi_2)$ under the $\mu$ action of $\widetilde{\pr}_{g,n}$.

To study the number of orbits of the action of $\widetilde{\pr}_{g,n}$ on $\mathcal{P}_{g,n}$, we now exhibit fundamental domains for the $\lambda$-action of $\widetilde{\pr}_{g,n}$ on $\coprod_{d\in \mathbb{Z}} V_{g,n}^d$.  A fundamental domain $U$ is a subset of $\coprod_{d\in \mathbb{Z}} V_{g,n}^d$ that contains at least one point of each orbit of $\widetilde{\pr}_{g,n}$, and no two points in the interior of $U$ are equivalent. To state our result, we identify $V_{g,n}^0$ with $C_{g,n} \oplus D_{g,n}$ by means of Corollary~\ref{isoalphan}.
\begin{lemma} \label{fundamentaldomain} A fundamental domain for the action of $\picrel_{g,n}(\mathbb{Z})$ (obtained by restricting $\lambda$) on $\coprod_{d\in \mathbb{Z}} V_{g,n}^d$ is given by
\begin{enumerate}
\item any hypercube in $C_{g,n}$ of edge length $1$ when $g=0$;
\item the product of any hypercube in $C_{g,n}$ of edge length $1$ with the set \[\{d \in {1, \ldots, 2g-2} : \gcd(d+1-g, 2g-2)=1\}\] when  $g \geq 2$ and $n=0$;
\item the product of any hypercube in $C_{g,n}$ of edge length $1$ and of any $(2g-2+\delta_{1,g} \times 1 \times \ldots \times 1)$ top dimensional hyperrectangle in $D_{g,n}$, when $g,n \geq 1$.
    \end{enumerate}
\end{lemma}
In the proof we will use the free generators of the relative Picard group given in Fact~\ref{arbarellocornalba}. It will also be convenient to keep track of the degree in the right hand side of the isomorphism of Corollary~\ref{isoalphan}, thus we will write $V_{g,n}^d \cong C_{g,n}^d \oplus D_{g,n}^d$
\begin{proof}
A fundamental domain for the subgroup $W_{g,n} \subseteq \widetilde{\pr}_{g,n}$ is given by the product of any hypercube in $C^0_{g,n}$ of edge length $1$ with $\coprod_{d \in \mathbb{Z}} D_{g,n}^d$.

 When $g =0$ the claim follows, because the action of the subgroup $\picrel_{g,n}(\mathbb{Z})$ generated by a section on the collection of points $\coprod_{d \in \mathbb{Z}} D_{0,n}^d$ is transitive.

 When $g\geq 1$, a fundamental domain for the action of the subgroup generated by $\omega_{\pi}$ (or by $\Sigma_1$ when $g=1$) on $\coprod_{d \in \mathbb{Z}} D_{g,n}^d$ is given by $\coprod_{d=0}^{2g-3} D_{g,n}^d$.

 When $n=0$ the set $D_{g,0}^d$ contains $1$ element, and by Remark~\ref{nomarkedpoints} the nondegenerate $d$'s are those that satisfy $\gcd(d+1-g, 2g-2) = 1$.

 When $n \geq 1$ we are left to study the action of the free rank $n- \delta_{1,g}$ abelian group generated by the sections (that are distinct from $\Sigma_1$ when $g=1$). If $\Sigma$ is any such section, translation by $\deg(\Sigma)$ identifies $D_{g,n}^d$ with $D_{g,n}^{d+1}$. Modulo $\omega_{\pi}$ (or modulo $\Sigma_1$ when $g=1$), translation by $(2g-2+ \delta_{1,g}) \cdot \Sigma_j$  identifies any point $(x_{1+\delta_{1,g}}, \ldots, x_n) \in D_{g,n}^{0}$ with $(x_{1+\delta_{1,g}}, \ldots, x_j + (2g-2+ \delta_{1,g}), \ldots,  x_n)$. This concludes the proof.
\end{proof}
The orbits of the $\mu$-action of $\widetilde{\pr}_{g,n}$ on $\mathcal{P}_{g,n}$ can be read off from the action of $\mathbb{Z}/2 \mathbb{Z} = \widetilde{\pr}_{g,n} / \picrel_{g,n}(\mathbb{Z})$ on the collection of polytopes in the fundamental domains that we exhibited in Lemma~\ref{fundamentaldomain}.

\begin{corollary} \label{whentransitive} The action $\mu$ of $\widetilde{\pr}_{g,n}$ on the set $\mathcal{P}_{g,n}$ of stability polytopes
\begin{enumerate}
\item has finitely many orbits;
\item is free if and only if $g\geq 3$, or $g$ and $n \geq 2$, or $g=1$ and $n \geq 3$;
\item is transitive if and only if $g=0$, or $(g,n)$ belongs to the set \begin{equation} \label{list} \{ (1,1), (1,2), (1,3), (2,0), (2,1), (3,0), (4,0)\}. \end{equation}
\end{enumerate}
\end{corollary}

\begin{proof} We apply Lemma~\ref{fundamentaldomain}, choosing the fundamental domain to be a hyperrectangle that is equal the union of (closed) stability polytopes (that this can be done follows from the equations of the stability subspaces of Theorem~\ref{stabilityspace}). This reduces our claims to studying the action of $\mathbb{Z}/2 \mathbb{Z}$ on the set of polytopes in this hyperrectangle. All three claims follow then from the explicit description of the stability subspaces given in  Theorem~\ref{stabilityspace}.
  \end{proof}

Here is the main result of this section.
\begin{corollary} \label{noisocommuting}
 For fixed $(g,n)$, there exist finitely many isomorphism classes over $\Mb_{g,n}$ of $\Jb_{g,n}(\phi)$ for all integers $d$ and all nondegenerate $\phi \in V_{g,n}^d$.	When $g>0$ and the pair $(g,n)$ does not belong to~\eqref{list},
 there exist nondegenerate $\phi_1 \in V_{g,n}^{e_1}$ and $\phi_2 \in V_{g,n}^{e_2}$ such that $\Jb_{g, n}(\phi_1)$ is not isomorphic to $\Jb_{g, n}(\phi_2)$ over $\Mb_{g, n}$.
\end{corollary}

\begin{proof}
By Lemma~\ref{key} and Lemma~\ref{isoiffsameorbit}, an isomorphism over $\Mb_{g,n}$ exists if and only $\mathcal{P}(\phi_1)$ and $\mathcal{P}(\phi_2)$ belong to the same orbit of the $\mu$-action of  $\widetilde{\pr}_{g,n}$ on $\mathcal{P}_{g,n}$. The first claim then follows from the first part of Corollary~\ref{whentransitive}, and the second claim follows from the third part of the same Corollary.
\end{proof}

When the coarse moduli scheme $\overline{M}_{g,n}$ is a variety of general type,  we can deduce that two $\phi$-compactified universal Jacobans $\Jb_{g, n}(\phi_1)$ and $\Jb_{g, n}(\phi_2)$ as in Corollary~\ref{noisocommuting} are in fact non-isomorphic as Deligne--Mumford stacks (and not just over $\Mb_{g,n}$). To prove this, we will employ the following lemma, in which we exploit the birational uniqueness of the Iitaka fibration, arguing similarly to \cite[Theorem~7.3]{bfv}.
\begin{lemma} \label{isomustcommute}
If the Kodaira dimension $\kappa(\overline{M}_{g,n})$ equals $3g-3+n$, any isomorphism $\Jb_{g, n}(\phi_1) \to \Jb_{g, n}(\phi_2)$ commutes with the forgetful map to $\Mb_{g,n}$ up to an automorphism that permutes the marked points.
\end{lemma}
\begin{proof}
We claim that the Kodaira dimension $\kappa (\overline{J}_{g, n}(\phi_i))$ equals $3g-3+n$. In view of Iitaka's easy addition inequality (\cite[Theorem~6.12]{ueno}), we have \[\kappa(\overline{J}_{g, n}(\phi_i)) \leq  \dim(\overline{M}_{g,n}) + \kappa (\pi^{-1}([C, p_i])) \] for a general curve $(C, p_1, \ldots, p_n)$ of $\Mb_{g,n}$. The reverse inequality
\[
\kappa(\overline{J}_{g, n}(\phi_i)) \geq \kappa(\overline{M}_{g,n}) + \kappa (\pi^{-1}([C, p_i])) \quad \textrm{for } (C,p_1, \ldots, p_n) \textrm { general in } \Mb_{g,n}
\]
follows from the Iitaka conjecture for abelian varieties (the main result of \cite{ueno2}). Since we are assuming that $\kappa(\overline{M}_{g,n}) = \dim(\overline{M}_{g,n})$, the claim follows.

The forgetful morphism of coarse moduli schemes $p \colon \overline{J}_{g, n}(\phi_i) \to \overline{M}_{g,n}$ is an algebraic fibration (i.e.~it is surjective and with geometrically connected fibers) of normal varieties with $\kappa (\overline{J}_{g, n}(\phi_i)) = \dim (\overline{M}_{g,n})$,  and the Kodaira dimension of a general fiber of $p$ equals zero, so $p$ is the Iitaka fibration by \cite[Theorem~6.11]{ueno}.

Since the Iitaka fibration is a birational invariant, any isomorphism $\alpha \colon \overline{J}_{g, n}(\phi_1) \to \overline{J}_{g, n}(\phi_2)$ induces a birational map $\beta$ such that the diagram
\begin{displaymath}
	\xymatrix{
\overline{J}_{g, n}(\phi_1) \ar[d] \ar[r]^{\alpha} &\overline{J}_{g, n}(\phi_2)\ar[d]\\
\overline{M}_{g,n} \ar@{-->}
[r]^{\beta}          &\overline{M}_{g,n}}
\end{displaymath}
commutes. To conclude, it is enough to show that $\beta$ extends to an automorphism of $\Mb_{g,n}$ that lifts to an automorphism of $\Jb_{g,n}(\phi_i)$ for $i=1,2$.

The birational map $\beta$ induces a rational map $\overline{M}_{g,0} \dasharrow \overline{M}_{g,0}$, which is the identity by \cite[Lemma 7.4]{bfv}. Therefore, if $C$ is a general curve of $\Mb_{g,0}$, the birational map $\beta$ induces an automorphism of the Fulton-MacPherson compactification $C[n]$ of the configuration space of $n$ points on $C$ (the fiber of $[C]$ under the forgetful map). By \cite[Proposition 4.11]{alex}, the automorphism group of $C[n]$ is the symmetric group on $n$ elements. We deduce that $\beta$ is the automorphism of $\Mb_{g,n}$ induced by a certain permutation of the marked points, and as such it lifts to an automorphism of $\Jb_{g,n}(\phi_i)$. \end{proof}

We conclude this section with the following corollary.

\begin{corollary} \label{nonisomorphic} When the pair $(g,n)$ is such that $\overline{M}_{g,n}$ is of general type, there exist nondegenerate $\phi_1, \phi_2$  such that $\Jb_{g,n}(\phi_1)$  and $\Jb_{g,n}(\phi_2)$ are non-isomorphic.
\end{corollary}

\begin{proof}
It is well-known that when $g=0$ and when $(g,n)$ belongs to~\eqref{list}, the moduli scheme $\overline{M}_{g,n}$ is uniruled, in particular it is not of general type. By combining this observation with Corollary~\ref{noisocommuting} and Lemma~\ref{isomustcommute}, we deduce the statement.
\end{proof}

\begin{remark} We believe that the problem of determining all pairs $(g,n)$ such that $\overline{M}_{g,n}$ is of general type is still open. A well-known sufficient condition for $\overline{M}_{g,n}$ to be of general type is that $g \geq 24$. This was proven by Eisenbud-Harris-Mumford in \cite{eh} and \cite{hm} when $n=0$.  For $n>0$ this follows from loc.~cit.~and from the Iitaka conjecture for curve fibrations, which was proven by Viehweg in \cite{vie}.
\end{remark}

\begin{remark} \label{fixedd} The main results of this section are concerned with the problem of classifying isomorphism classes of fine compactified universal Jacobians $\Jb_{g,n}(\phi)$ for varying $\phi$ in $V_{g,n}^d$ and varying $d \in \mathbb{Z}$. The classification problem is similar when $d$ is  fixed.

More precisely, we claim that the statements of Corollaries~\ref{noisocommuting} and~\ref{nonisomorphic} remain valid for fixed $d$ provided $n \geq 1$. (When $n=0$, for fixed $d$ there is $1$ fine compactified universal $\Jb_{g,n}(\phi)$ when $\gcd(d+1-g, 2g-2)=1$ and there is none otherwise, see Remark~\ref{nomarkedpoints}).

%A  similar problem is to fix $d \in \mathbb{Z}$ and classify all different $\Jb_{g,n}(\phi)$ for $\phi \in V_{g,n}^d$ nondegenerate, modulo isomorphisms over $\Mb_{g,n}$. This admits a very similar combinatorial solution.

Indeed, let $\widetilde{\pr}_{g,n}^{0}$ be the stabilizer of the set $V_{g,n}^0$ for the $\lambda$-action of $\widetilde{\pr}_{g,n}$ on $V_{g,n}$. Then the $\mu$-action of $\widetilde{\pr}_{g,n}$ on the collection of all stability polytopes $\mathcal{P}_{g,n}$ restricts to an action $\mu_d$ of $\widetilde{\pr}_{g,n}^0$ on the collection $\mathcal{P}^d_{g,n}$ of stability polytopes in $V_{g,n}^d$. We have already observed that, as a consequence of Lemma~\ref{isoiffsameorbit}, the quotient set $\mathcal{P}_{g,n}/ \widetilde{\pr}_{g,n}$ classifies fine compactified universal Jacobians $\Jb_{g,n}(\phi)$ modulo isomorphisms over $\Mb_{g,n}$ for an arbitrary total degree $d$. By the same argument, for every fixed $d \in \mathbb{Z}$ the quotient set $\mathcal{P}^d_{g,n}/ \widetilde{\pr}^0_{g,n}$ classifies degree-$d$ fine compactified universal Jacobians $\Jb_{g,n}(\phi)$ modulo isomorphisms over $\Mb_{g,n}$.

There is a well-defined inclusion \[\tau \colon \mathcal{P}^d_{g,n}/\widetilde{\pr}^0_{g,n} \to \mathcal{P}_{g,n} / \widetilde{\pr}_{g,n}.\] Our claim follows from the fact that, when $n \geq 1$, the map $\tau$ is surjective.  Indeed, let $e \in \mathbb{Z}$ and $\phi \in V_{g,n}^e$ nondegenerate. The polytope $\mathcal{P}(\phi) \subset V_{g,n}^e$ is equal to the translate of the polytope \[\mathcal{P}(\phi - (e-d) \deg(\Sigma)) \subset V_{g,n}^d\] by $(e-d) \deg(\Sigma)$, for $\deg(\Sigma)$ the multidegree of any section $\Sigma \colon \Mb_{g,n} \to \Cb_{g,n}$. This concludes the proof of surjectivity and of our claim.
\end{remark}

\section{Appendix: Properties of reflexive sheaves} \label{Section: Appendix}
Here we collect some results about reflexive sheaves that we use in  Section~\ref{final}.  The conditions $G_n$, $R_n$, and $S_n$ we discuss are taken from \cite{hartshorne}.

\begin{lemma} \label{Lemma: CommAlgebra}
	Let $f \colon \mathcal{X} \to \mathcal{S}$ be a family of curves over a regular Deligne--Mumford stack $\mathcal{S}$ and $F$ a family of torsion-free sheaves on $\mathcal{X}$.  If $\mathcal{X}$ satisfies conditions $G_1$ and $S_2$, then $F$ is reflexive.
\end{lemma}
\begin{proof}
	By \cite[Theorem~1.9]{hartshorne}, it is enough to show that $F$ satisfies $S_2$.  In other words, we need to show that if $x \in \mathcal{X}$, then $\operatorname{depth} F_{x} \ge \operatorname{min}(2, \operatorname{dim} \mathcal{O}_{\mathcal{X}, x})$. Given $x$, set $s := f(x)$.  By hypothesis, $\mathcal{O}_{\mathcal{S}, s}$ is regular so its maximal ideal $\mathfrak{m}_{s}$ is generated by a regular sequence $a_1, \ldots, a_d$ of length $d := \operatorname{dim} \mathcal{O}_{S, s}$.  The images $f^{*}(a_1), \ldots, f^{*}(a_d) \in \mathcal{O}_{X, x}$ are regular on $F_{x}$ by flatness, and the quotient module $F_{x}/ f^{*}(a_1) \cdot F_{x} + \ldots + f^{*}(a_d) \cdot F_{x}$ is torsion-free by hypothesis.  Pick an element $b \in \mathcal{O}_{X, x}$ that acts as a nonzero divisor on this quotient module.  Then $f^{*}(a_1), \ldots, f^{*}(a_d), b$ is a $F_{x}$-regular sequence, so $\operatorname{depth} F_{x} \ge d+1 = \dim \mathcal{O}_{\mathcal{X}, x}$.
\end{proof}

\begin{corollary} \label{Cor: determinedcod2}
	With the  hypothesis of Lemma~\ref{Lemma: CommAlgebra}, if $G$ is a second family of rank~$1$ torsion-free sheaves on $\mathcal{X}$ and $\mathcal{Y} \subsetneq \mathcal{X}$ is a closed substack of codimension $\ge 2$ such that $F|_{\mathcal{X} - \mathcal{Y}}$ is isomorphic to $G|_{\mathcal{X}-\mathcal{Y}}$, then $F$ is isomorphic to $G$.
\end{corollary}
\begin{proof}
 	This is a special case of \cite[Theorem~1.12]{hartshorne}. (The result is stated for schemes, and we deduce the statement for stacks by passing to an \'{e}tale cover.)
\end{proof}

\begin{lemma} \label{Lemma: G1S2}
	For any nondegenerate $\phi \in V_{g, n}^d$, the fiber product $\Jb_{g,n}(\phi) \times_{\Mb_{g,n}} \Cb_{g,n}$ satisfies $G_1$ and $S_2$.
\end{lemma}
\begin{proof}
	We will prove the stronger result that the fiber product is Cohen--Macaulay and satisfies $R_1$.  Certainly $\Jb(\phi) \times_{\Mb_{g,n}} \Cb_{g,n}$ is regular at every pair consisting of a line bundle and a point that is not a node (since $\Jb_{g, n}(\phi) \to \Mb_{g,n}$ is smooth at a line bundle and $\Cb_{g, n} \to \Mb_{g, n}$ is smooth at a point that is not a node).  The locus of such pairs has codimension $1$, so we conclude that the fiber product satisfies $R_1$.
	
	To complete the proof, observe that the deformation theory argument in \cite[Lemma~3.33]{kasspa} shows the completed local ring of the fiber product is a power series ring over a ring that is the completed tensor product of rings of the form $k[[x, y, u, v]]/ x y - u v$ or $k[[t]]$.   In particular, the completed local ring is a power series ring over a complete intersection ring and hence is Cohen--Macaulay, i.e.~satisfies $S_{d}$ for all $d$.
\end{proof}

Finally we show that, on a nodal curve, the rule sending a rank~$1$ torsion-free sheaf $F$ to its dual $F^{\vee}$  commutes with base change and hence defines a isomorphism $\Jb_{g, n}(\phi) \to \Jb_{g, n}(-\phi)$.

\begin{lemma} \label{basechange}
	Let $F$ be a family of rank~$1$ torsion-free sheaves on a family $C \to S$ of nodal curves.  Then the formulation of the dual $F^{\vee} := \Hom( F, \mathcal{O}_{C})$ commutes with base change.  In other words, if $T \to S$ is a $k$-morphism, then the natural map
	\[
		F^{\vee} \otimes \mathcal{O}_{C_T} \to (F \otimes \mathcal{O}_{C_T})^{\vee}
	\]
	is an isomorphism.
\end{lemma}
\begin{proof}
	By \cite[Theorem~1.10]{altman80}, it is enough to show that $\operatorname{Ext}^{1}(I_s, \mathcal{O}_{X_s})$ vanishes for every point $s \in S$, and because $X_s$ is Gorenstein, vanishing is a special case of \cite[Proposition~6.1]{hartshorne}.
\end{proof}

\section{Acknowledgements}
The authors would like to thank  the anonymous referee, Renzo Cavalieri, Barbara Fantechi, Samuel Grushevsky, David Holmes, Klaus Hulek, Margarida Melo, Tif Shen, Mattia Talpo, Nicola Tarasca,  Orsola Tommasi, and Dmitry Zakharov for useful feedback, discussions and comments.

Many thanks to Filippo Viviani for catching an error in our statement of Fact~1 in a previous version of this paper.

Jesse Leo Kass was  supported by a grant from the Simons Foundation  (Award Number 429929) and by the National Security Agency under Grant Number H98230-15-1-0264.  The United States Government is authorized to reproduce and distribute reprints notwithstanding any copyright notation herein. This manuscript is submitted for publication with the understanding that the United States Government is authorized to reproduce and distribute reprints.

Nicola Pagani was supported by the EPSRC First Grant EP/P004881/1 with title  ``Wall-crossing on universal compactified Jacobians''.

\bibliographystyle{amsalpha}

\bibliography{biblio}

\providecommand{\bysame}{\leavevmode\hbox to3em{\hrulefill}\thinspace}
\providecommand{\MR}{\relax\ifhmode\unskip\space\fi MR }
\providecommand{\MRhref}[2]{%
  \href{http://www.ams.org/mathscinet-getitem?mr=#1}{#2}
}
\providecommand{\href}[2]{#2}
\begin{thebibliography}{CMKV15}

\bibitem[AC87]{acpicard}
Enrico Arbarello and Maurizio Cornalba, \emph{The {P}icard groups of the moduli
  spaces of curves}, Topology \textbf{26} (1987), no.~2, 153--171. \MR{895568}

\bibitem[AC98]{ac}
\bysame, \emph{Calculating cohomology groups of moduli spaces of curves via
  algebraic geometry}, Inst. Hautes \'Etudes Sci. Publ. Math. (1998), no.~88,
  97--127 (1999). \MR{1733327}

\bibitem[ACG11]{acg2}
Enrico Arbarello, Maurizio Cornalba, and Phillip~A. Griffiths, \emph{Geometry
  of algebraic curves. {V}olume {II}}, Grundlehren der Mathematischen
  Wissenschaften [Fundamental Principles of Mathematical Sciences], vol. 268,
  Springer, Heidelberg, 2011, With a contribution by Joseph Daniel Harris.
  \MR{2807457}

\bibitem[AK80]{altman80}
Allen~B. Altman and Steven~L. Kleiman, \emph{Compactifying the {P}icard
  scheme}, Adv. in Math. \textbf{35} (1980), no.~1, 50--112. \MR{555258}

\bibitem[Ale04]{alexeev04}
Valery Alexeev, \emph{Compactified {J}acobians and {T}orelli map}, Publ. Res.
  Inst. Math. Sci. \textbf{40} (2004), no.~4, 1241--1265. \MR{2105707}

\bibitem[BFV12]{bfv}
Gilberto Bini, Claudio Fontanari, and Filippo Viviani, \emph{On the birational
  geometry of the universal {P}icard variety}, Int. Math. Res. Not. IMRN
  (2012), no.~4, 740--780. \MR{2889156}

\bibitem[BL92]{bl}
Christina Birkenhake and Herbert Lange, \emph{Complex abelian varieties},
  Grundlehren der Mathematischen Wissenschaften [Fundamental Principles of
  Mathematical Sciences], vol. 302, Springer-Verlag, Berlin, 1992. \MR{1217487}

\bibitem[Cap94]{caporaso}
Lucia Caporaso, \emph{A compactification of the universal {P}icard variety over
  the moduli space of stable curves}, J. Amer. Math. Soc. \textbf{7} (1994),
  no.~3, 589--660. \MR{1254134}

\bibitem[Cap08]{caporaso08a}
\bysame, \emph{N\'eron models and compactified {P}icard schemes over the moduli
  stack of stable curves}, Amer. J. Math. \textbf{130} (2008), no.~1, 1--47.
  \MR{2382140}

\bibitem[CJM11]{cjm}
Renzo Cavalieri, Paul Johnson, and Hannah Markwig, \emph{Wall crossings for
  double {H}urwitz numbers}, Adv. Math. \textbf{228} (2011), no.~4, 1894--1937.
  \MR{2836109}

\bibitem[CMKV15]{cmkv}
Sebastian Casalaina-Martin, Jesse~Leo Kass, and Filippo Viviani, \emph{The
  local structure of compactified {J}acobians}, Proc. Lond. Math. Soc. (3)
  \textbf{110} (2015), no.~2, 510--542. \MR{3335286}

\bibitem[Dud18]{dudin}
Bashar Dudin, \emph{Compactified universal {J}acobian and the double
  ramification cycle}, Int. Math. Res. Not. IMRN (2018), no.~8, 2416--2446.
  \MR{3801488}

\bibitem[EH87]{eh}
David Eisenbud and Joe Harris, \emph{The {K}odaira dimension of the moduli
  space of curves of genus {$\geq 23$}}, Invent. Math. \textbf{90} (1987),
  no.~2, 359--387. \MR{910206}

\bibitem[EP16]{espac}
Eduardo Esteves and Marco Pacini, \emph{Semistable modifications of families of
  curves and compactified {J}acobians}, Ark. Mat. \textbf{54} (2016), no.~1,
  55--83. \MR{3475818}

\bibitem[Est01]{esteves}
Eduardo Esteves, \emph{Compactifying the relative {J}acobian over families of
  reduced curves}, Trans. Amer. Math. Soc. \textbf{353} (2001), no.~8,
  3045--3095. \MR{1828599}

\bibitem[GP03]{graberpanda}
T.~Graber and R.~Pandharipande, \emph{Constructions of nontautological classes
  on moduli spaces of curves}, Michigan Math. J. \textbf{51} (2003), no.~1,
  93--109. \MR{1960923}

\bibitem[GZ14]{grushevsky}
Samuel Grushevsky and Dmitry Zakharov, \emph{The zero section of the universal
  semiabelian variety and the double ramification cycle}, Duke Math. J.
  \textbf{163} (2014), no.~5, 953--982. \MR{3189435}

\bibitem[Har94]{hartshorne}
Robin Hartshorne, \emph{Generalized divisors on {G}orenstein schemes},
  Proceedings of {C}onference on {A}lgebraic {G}eometry and {R}ing {T}heory in
  honor of {M}ichael {A}rtin, {P}art {III} ({A}ntwerp, 1992), vol.~8, 1994,
  pp.~287--339. \MR{1291023}

\bibitem[HM82]{hm}
Joe Harris and David Mumford, \emph{On the {K}odaira dimension of the moduli
  space of curves}, Invent. Math. \textbf{67} (1982), no.~1, 23--88, With an
  appendix by William Fulton. \MR{664324}

\bibitem[Hol17]{holmes}
David Holmes, \emph{Extending the double ramification cycle by resolving the
  {A}bel--{J}acobi map}, arXiv:1707.02261, 2017.

\bibitem[Kou91]{kouvi}
Alexis Kouvidakis, \emph{The {P}icard group of the universal {P}icard varieties
  over the moduli space of curves}, J. Differential Geom. \textbf{34} (1991),
  no.~3, 839--850. \MR{1139648}

\bibitem[KP17]{kasspa}
Jesse~Leo Kass and Nicola Pagani, \emph{Extensions of the universal theta
  divisor}, Adv. Math. \textbf{321} (2017), no.~1, 221--268.

\bibitem[Mas16]{alex}
Alex Massarenti, \emph{On the biregular geometry of {F}ulton--{M}ac{P}herson
  configuration spaces}, https://arxiv.org/abs/1603.06991, 2016.

\bibitem[Mel09]{melo09}
Margarida Melo, \emph{Compactified {P}icard stacks over
  {$\overline{\mathcal{M}}_g$}}, Math. Z. \textbf{263} (2009), no.~4, 939--957.
  \MR{2551606}

\bibitem[Mel11]{melo11}
\bysame, \emph{Compactified {P}icard stacks over the moduli stack of stable
  curves with marked points}, Adv. Math. \textbf{226} (2011), no.~1, 727--763.
  \MR{2735773}

\bibitem[Mel16]{melo}
\bysame, \emph{Compactifications of the universal {J}acobian over curves with
  marked points}, https://arxiv.org/abs/1509.06177v3, 2016.

\bibitem[Mes87]{mestrano}
Nicole Mestrano, \emph{Conjecture de {F}ranchetta forte}, Invent. Math.
  \textbf{87} (1987), no.~2, 365--376. \MR{870734}

\bibitem[Mil86]{milne}
J.~S. Milne, \emph{Jacobian varieties}, Arithmetic geometry ({S}torrs, {C}onn.,
  1984), Springer, New York, 1986, pp.~167--212. \MR{861976}

\bibitem[MW17]{marcus}
Steffen Marcus and Jonathan Wise, \emph{Logarithmic compactification of the
  {A}bel--{J}acobi section}, arXiv:1708.04471, 2017.

\bibitem[Nam76]{namikawa}
Yukihiko Namikawa, \emph{A new compactification of the {S}iegel space and
  degeneration of {A}belian varieties. {I}}, Math. Ann. \textbf{221} (1976),
  no.~3, 201--241.

\bibitem[OS79]{oda79}
Tadao Oda and C.~S. Seshadri, \emph{Compactifications of the generalized
  {J}acobian variety}, Trans. Amer. Math. Soc. \textbf{253} (1979), 1--90.
  \MR{536936}

\bibitem[Pan96]{panda}
Rahul Pandharipande, \emph{A compactification over {$\overline {M}_g$} of the
  universal moduli space of slope-semistable vector bundles}, J. Amer. Math.
  Soc. \textbf{9} (1996), no.~2, 425--471. \MR{1308406}

\bibitem[Sim94]{simpson}
Carlos~T. Simpson, \emph{Moduli of representations of the fundamental group of
  a smooth projective variety. {I}}, Inst. Hautes \'Etudes Sci. Publ. Math.
  (1994), no.~79, 47--129. \MR{1307297}

\bibitem[SSV08]{ssv}
S.~Shadrin, M.~Shapiro, and A.~Vainshtein, \emph{Chamber behavior of double
  {H}urwitz numbers in genus 0}, Adv. Math. \textbf{217} (2008), no.~1, 79--96.
  \MR{2357323}

\bibitem[Tha96]{thaddeus}
Michael Thaddeus, \emph{Geometric invariant theory and flips}, J. Amer. Math.
  Soc. \textbf{9} (1996), no.~3, 691--723. \MR{1333296}

\bibitem[Uen75]{ueno}
Kenji Ueno, \emph{Classification theory of algebraic varieties and compact
  complex spaces}, Lecture Notes in Mathematics, Vol. 439, Springer-Verlag,
  Berlin-New York, 1975, Notes written in collaboration with P. Cherenack.
  \MR{0506253}

\bibitem[Uen78]{ueno2}
\bysame, \emph{On algebraic fibre spaces of abelian varieties}, Math. Ann.
  \textbf{237} (1978), no.~1, 1--22. \MR{506652}

\bibitem[Vie77]{vie}
Eckart Viehweg, \emph{Canonical divisors and the additivity of the {K}odaira
  dimension for morphisms of relative dimension one}, Compositio Math.
  \textbf{35} (1977), no.~2, 197--223. \MR{0569690}

\end{thebibliography}

\end{document}